\documentclass[a4j,11pt]{article} 

\usepackage{amssymb,amsthm}
\usepackage{amsmath}
\newtheorem{Thm}{Theorem}[section]
\newtheorem{Lem}{Lemma}
\newtheorem{Prop}{Proposition}

\usepackage{color}

\newcommand{\R}{\mathbb{R}}

\newcommand{\wast}{\mbox{weak}^\ast}

\renewcommand{\epsilon}{\varepsilon}

\newcommand{\ydt}{dy(\tau)}
\newcommand{\ydtp}{dy^{+}(\tau)}

\renewcommand{\phi}{\varphi}

\newcommand{\M}{\mathcal{M}}

\def\ru[#1][#2]{{#1}^{#2}}
\def\rl[#1][#2]{{#1}_{#2}}

\newcommand{\dx}{\Delta x}

\newcommand{\dV}{\Delta V}

\makeatletter
\@addtoreset{equation}{section}

\makeatother

\title{
An interior point sequential quadratic programming-type method 
for log-determinant semi-infinite programs
\thanks{This work was supported by JSPS KAKENHI Grant Number [15K15943].
}
}

\author{Takayuki Okuno
\thanks{RIKEN, The Center for Advanced Intelligence Project (AIP), Nihonbashi 1-chome Mitsui Building, 15th floor,1-4-1 Nihonbashi, Chuo-ku, Tokyo 103-0027, Japan,
E-mail: takayuki.okuno.ks@riken.jp}\hspace{2em}
Masao Fukushima
\thanks{Nanzan University, Faculty of Science and Engineering, 18 Yamazato-cho, Showa-ku, Nagoya 466-8673, Japan,
E-mail: {fuku@nanzan-u.ac.jp}}
}

\begin{document}
\maketitle
\begin{abstract}
In this paper, we consider a nonlinear semi-infinite program that minimizes a 
function including a log-determinant (logdet) function over 
positive definite matrix constraints and infinitely many convex inequality constraints, called SIPLOG for short. 
The main purpose of the paper is to develop an algorithm for computing a Karush-Kuhn-Tucker (KKT) point
for the SIPLOG efficiently.
More specifically, we propose an interior point sequential quadratic programming-type method that 
inexactly solves a sequence of semi-infinite quadratic programs 
approximating the SIPLOG.
Furthermore, to generate a search direction in the dual matrix space associated with the semi-definite constraint, we solve 
{scaled Newton equations {that yield} the family of Monteiro-Zhang directions}.
We prove that 
the proposed method 
weakly* converges to a KKT point under some mild assumptions. Finally, we conduct some numerical experiments to demonstrate the efficiency of the proposed method.
\end{abstract}





\section{Introduction}
In this paper, we consider the following semi-infinite 
program that minimizes 
a nonlinear function including a log-determinant (logdet) function over
an infinite number of convex inequality constraints and a positive-semidefinite constraint, called SIPLOG for short:
{\begin{align}
 \begin{array}{ll}
  \displaystyle{\mathop{\rm Minimize}}  &f(x)-\mu\log\det F(x)
  \vspace{0.5em}\\
  {\rm subject~to} &g(x,\tau)\le 0\ \mbox{ for all } {\tau}\in T, \\
                   &F(x)\in S^m_{++},\\
                   &Gx=h,   
 \end{array}\label{lsisdp}
\end{align}
where $\mu\in \R$
is a positive constant,
$f:\R^n\to \R$ is a continuously differentiable function, and $T$ is a compact metric space.
In addition, $g:\R^n\times T\to \R$ is a continuous function with $g(\cdot,\tau)$ being convex and continuously differentiable. 
Moreover, $S^m$ and $S^m_{++} (S^m_{+})$ denote the sets of 
$m\times m$ symmetric matrices 
and symmetric positive (semi-)definite matrices, respectively, and 
$F(\cdot):\R^n\to S^{m}$ is an affine function, i.e.,
$$
F(x):=F_0+\sum_{i=1}^nx_iF_i
$$
with $F_i\in S^m$ for $i=0,1,\ldots,n$ and $x=(x_1,x_2,\ldots,x_n)^{\top}$.
Finally, $G\in \R^{s\times n}$ and $h\in \R^s$.

{Throughout the paper, 
we assume that SIPLOG\,\eqref{lsisdp} has a Slater point, i.e., a vector $\bar{x}\in \R^n$ such that
$$
F(\bar{x})\in S^m_{++},\ g(\bar{x},\tau)<0\ (\tau\in T),\ 
G\bar{x}=h.
$$
This assumption implies that the set of feasible points taking finite objective values is not empty.

When $\mu{\searrow} 0$, the SIPLOG reduces to a semi-infinite semi-definite program (SISDP):
\begin{equation}
{\rm min}\ f(x)\ \ \mbox{s.t. }\ g(x,\tau)\le 0\ \mbox{ for all } {\tau}\in T,\ {F}(x)\in S^m_{+},\ Gx=h.\label{lsisdp2}
\end{equation}
FIR filter design problem\,\cite{spwu1996} and
robust envelop-constrained filter design with orthonormal bases\,\cite{li2007robust} can be formulated as an SISDP whose functions are all affine with respect to $x$. 
For solving linear SISDPs, a discretization-type method and relaxed cutting plane method were proposed
in \cite{li2004solution} and \cite{li2006relaxed}, respectively.
In the absence of the semi-definite constraint and the logdet function, the SIPLOG becomes a nonlinear semi-infinite program (NSIP) with an infinite number of convex inequality constraints.  For an overview of the NSIP, see \cite{sip2,sip-recent,Reem} and references therein. 
On the other hand, in the absence of the semi-infinite constraints, the SIPLOG becomes a nonlinear semi-definite program (SDP).  
We can find 
some important applications for the nonlinear SDP in finance\,\cite{leibfritz2009successive,konno2003cutting}
and optimal control\,\cite{freund2007nonlinear,leibfritz2002interior}.
For solving the nonlinear SDP,  several existing algorithms for nonlinear programs such as a primal dual interior point method and a sequential quadratic programming (SQP) method were extended\,\cite{freund2007nonlinear,yabe,yamashita2012local}. See the survey article\,\cite{yamashita2015survey} for other algorithms of the nonlinear SDP. 

The logdet function plays a crucial role in 
various fields such as statistics, experimental design, and information and communication theory.
In continuous optimization, it has a close connection with the interior point method for the SDP\,\cite{wolkowicz2012handbook}.
Accordingly, many algorithms 
for solving optimization problems including the logdet function
have been studied extensively so far. For example, see \cite{vandenberghe1998determinant,yang2013proximal,wang2016solve}.

It makes sense to study the SIPLOG itself.  
Indeed, the D-optimal experimental design problem can be formulated as an SIPLOG straightforwardly\,\cite{vandenberghe1998connections}.  
Moreover, 
in the spirit of the primal-dual interior point method for the nonlinear SDP\,\cite{yabe}, 
we can expect that a sequence of Karush-Kuhn-Tucker (KKT) points for the SIPLOG
with $\mu>0$ decreasing to 0 converges to a KKT point of the SISDP\,\eqref{lsisdp2}. 
Hence, development of algorithms for solving the SIPLOG can be connected to efficient interior point methods for the SISDP.

In this paper, we focus on computing a KKT point of the SIPLOG\,\eqref{lsisdp}.
More specifically, we propose a new interior point SQP-type algorithm combined with an exchange method\,\cite{
Reem, Lai, okuno2012regularized,okuno2016exchange}.
In the method, we inexactly solve a semi-infinite (convex) quadratic program approximating the SIPLOG\,\eqref{lsisdp} to compute a search direction in the primal space.
Furthermore, to compute a search direction in the dual matrix space associated with the semi-definite constraint $F(x)\in S^m_+$, 
we solve certain 
scaled Newton equations that yield
the family of Monteiro-Zhang directions\,\cite[Chapter~10]{wolkowicz2012handbook}.
The proposed method can be regarded as an extension of the primal-dual interior point method \cite{yabe} for computing a barrier KKT point for the nonlinear SDP. However, the extension is not straightforward due to the presence of semi-infinite constraints.} 

The paper is organized as follows: 
In Section~\ref{sec:2}, we describe the KKT conditions for the SIPLOG.
In Section~\ref{sec:3}, we
propose an interior point SQP-type method for finding a KKT point and establish its convergence. 
In Section~\ref{sec:4}, we conduct some numerical experiments to demonstrate the efficiency of the proposed method.
Finally, we conclude the paper with some remarks.

\subsection*{Notations}
The identity matrix of order $m$ is denoted by $I$.
For any $P\in \R^{m\times m}$, ${\rm Tr}(P)$ denotes the trace of $P$.
For any symmetric matrices $X,Y\in S^m$, we define
the Jordan product of $X$ and $Y$ by $X\circ Y:=(XY+YX)/2$ and the inner product of $X$ and $Y$ by 
$X\bullet Y={\rm Tr}(XY)$. 
Also, we denote the Frobenius norm of $X\in S^m$ by $\|X\|_F:=\sqrt{X\bullet X}$.
We define the linear operator $\mathcal{L}_X:S^m\to S^m$ by $\mathcal{L}_X(Z):=X\circ Z$ for any $X\in S^m$.
We also denote $({\zeta})_+:=\max({\zeta},0)$ for any $\zeta\in \R$.
For sequences $\{y^k\}$ and $\{z^k\}$ of vectors, if $\|y^k\|\le M\|z^k\|$
for any $k$ with some $M>0$,
we write $\|y^k\|=O(\|z^k\|)$. 
Moreover, if there exists a positive sequence $\{\alpha_k\}$ with $\lim_{k\to\infty}\alpha_k=0$ and $\|y^k\|\le \alpha_k\|z^k\|$ {for any $k$},
we write $\|y^k\|=o(\|z^k\|)$. 
{For matrices $X_1,X_2,\ldots,X_p\in S^m$ and $Y\in S^m$,
we denote
$
(X_{i}\bullet Y)_{i=1}^n
:=
(X_1\bullet Y,
X_2\bullet Y,
\ldots,
X_{n}\bullet Y)^{\top}.
$}

{Let $\mathcal{C}(T)$ be the set of real-valued continuous functions defined on $T$ endowed with the supremum norm $\|h\|:=\max_{\tau\in T}|h(\tau)|$. 
Let $\mathcal{M}(T)$ be the dual space of $\mathcal{C}(T)$ 
that can be identified with the space of (finite signed) regular Borel measures with the Borel sigma algebra $\mathcal{B}$ on $T$ equipped with the total variation norm, i.e.,
$\|y\|:=\sup_{A\in \mathcal{B}}y(A)-\inf_{A\in \mathcal{B}}y(A)$ for $y\in \mathcal{M}(T)$, 
and denote by $\mathcal{M}_+(T)$ the set of all the nonnegative Borel measures of $\mathcal{M}(T)$. }
{\section{KKT conditions for the SIPLOG}\label{sec:2}
In this section, we present the Karush-Kuhn-Tucker (KKT) conditions for the SIPLOG\,\eqref{lsisdp}.
We say that the KKT conditions for SIPLOG\,\eqref{lsisdp} hold at $x^{\ast}\in \R^n$ if 
there exists some finite Borel-measure $y\in \M(T)$ such that 
\begin{align*}
&\nabla f(x^{\ast})+\int_T\nabla_xg(x^{\ast},\tau)\ydt-
(F_{i}\bullet \mu F(x^{\ast})^{-1})_{i=1}^n
+G^{\top}z=0,\notag \\ 
&\int_Tg(x^{\ast},\tau)\ydt=0,\ g(x^{\ast},\tau)\le 0\ (\tau\in T),\ y\in\M_+(T),\\
&Gx^{\ast}=h, \notag 
\end{align*}
where 
$z\in \R^s$ is a Lagrange multiplier vector {associated with} the equality constraints $Gx=h$. 
If $x^{\ast}$ is a local optimum of the SIPLOG, under Slater's constraint qualification, the KKT conditions hold at $x^{\ast}$.
In particular, there exists some discrete measure $y\in \M_+(T)$ satisfying the KKT conditions and $|{\rm supp}(y)|\le n$, where ${\rm supp}(y):=\{\tau\in T\mid y(\{\tau\})\neq 0\}$.
Conversely, when $f$ is convex, if the KKT conditions hold at $x^{\ast}$, then $x^{\ast}$ is an optimum of SIPLOG\,\eqref{lsisdp}.

Let $V\in S^m$.
Since $F(x^{\ast})\in S^m_{++}$,  
$\mu F(x^{\ast})^{-1} = V$ if and only if $F(x^{\ast})\circ V=\mu I$ and $V\in S^m_{++}$.
Then, using the matrix $V$ as a slack matrix variable, we can rewrite the KKT conditions as 
\begin{align}
&\nabla f(x^{\ast})+\int_T\nabla_xg(x^{\ast},\tau)\ydt-(F_i\bullet V)_{i=1}^n+G^{\top}z=0,\label{e1}\\ 
  &{F}(x^{\ast})\circ V=\mu I,\ F(x^{\ast})\in S^m_{++},\ V\in S^m_{++},\label{e2}\\
&\int_Tg(x^{\ast},\tau)\ydt=0,\ g(x^{\ast},\tau)\le 0\ (\tau\in T),\ y\in\M_+(T),\label{e3}\\
&Gx^{\ast}=h.\label{e4} 
\end{align}
Hereafter, we call a quadruple $(x,y,z,V)\in \R^n\times \M(T)\times \R^s\times S^m$ satisfying the conditions \eqref{e1}--\eqref{e4} a KKT point of the SIPLOG\,\eqref{lsisdp}. 

Note that the conditions\,\eqref{e1}, \eqref{e2}, and \eqref{e4} can be cast as the perturbed KKT conditions for the nonlinear SDP which is obtained by removing the semi-infinite constraint $g(x,\tau)\le 0\ (\tau\in T)$ from SISDP\,\eqref{lsisdp2}. 
Yamashita et al.\,\cite{yabe} proposed a primal-dual interior point method to find a solution satisfying those perturbed (barrier) KKT conditions for the nonliear SDP\footnote{
{Yamashita et al.\,\cite{yabe} considered the nonlinear SDP of the form:
$
\min\ f(x)\ \mbox{s.t. }\hat{h}(x)=0,\ \hat{G}(x)\in S^m_{++},
$
where the functions $\hat{h}:\R^n\to \R^s$ and $\hat{G}:\R^n\to S^{m}$ are continuously differentiable.} }.

In the next section, we will propose an interior point SQP-type algorithm for computing a KKT point $(x,y,z,V)$ satisfying the conditions\,\eqref{e1}--\eqref{e4}. 
This algorithm can be regarded as an extension of the algorithm proposed by Yamashita et al.\,\cite{yabe}.
Nevertheless, the way of extension is not straightforward because we must handle the semi-infinite constraint efficiently.}
\section{Interior point SQP-type algorithm for finding a KKT point}\label{sec:3}
In this section, we give an interior point SQP-type method for getting a KKT point of SIPLOG\,\eqref{lsisdp}.
Throughout the section, we use the following notations:
\begin{align*}
y_i^r&:=y(\tau_i^r)\ \ \mbox{for }{\rm supp}(y^r)=\{\tau_1^r,\tau_2^r,\ldots,\tau_{p_r}^r\}\\
w^r &:= (x^r,y^r,z^r,V_r)\in \R^n\times \mathcal{M}(T)\times \R^s\times S^m.
\end{align*}
The proposed algorithm composes iteration points $\{w^r\}_{r\ge 0}$ sequentially
by 
$$
\left(x^{r+1},
V_{r+1}
\right)
= \left(x^r +s_r\Delta x^r, V_{r}+s_r\Delta V_r\right),
$$
where 
$\left(\Delta x^r,\Delta V_r\right)\in \R^n\times S^m$ denotes a search direction and 
$s_r>0$ is a step size chosen so that the interior point constraints
\begin{equation*}
F(x^{r+1})\in S^m_{++}\ \mbox{and }\ V_{r+1}\in S^m_{++}
\end{equation*}
are satisfied.
In addition, we produce a sequence $\{y^r\}\subseteq \mathcal{M}_+(T)$ with $\left|{\rm supp}(y^r)\right|<\infty$ for any $r\ge 0$. 
Hereafter, we often drop super- or sub-scripts from those symbols for simplicity of expression.
\subsection{Search direction $(\Delta x,\Delta V)$ and Lagrange multipliers $(y^+, z^+)$}
In what follows, we explain how to generate a search direction $\left(\Delta x,\Delta V\right)$
together with Lagrange multiplier measure $y^+\in \M_+(T)$ and vector $z^+\in \R^s$
at the current point
$w=(x,y,z,V)$.

In applying an SQP-like method to \eqref{lsisdp}, it is natural to think of the following semi-infinite quadratic program, called SIQP for short, with infinitely many {\it linear} constraints:
\begin{align}
 \begin{array}{ll}
  \displaystyle{\mathop{\rm Minimize}_{\Delta x}}  &\nabla f(x)^{\top}\Delta x +\frac{1}{2}\Delta x^{\top}B\Delta x-\mu \xi(x)^{\top}\Delta x 
  \vspace{0.5em}\\
  {\rm subject~to} &g(x,\tau)+\nabla_xg(x,\tau)^{\top}\Delta x\le 0\hspace{1.0em}(\tau\in T),\\
                   &G(x+\Delta x)=h, 
 \end{array}\label{SIQP}
\end{align}
where the coefficient matrix $B\in S^n$ is chosen to be positive definite and 
the function $\xi:\R^n\to \R^n$ is defined by 
\begin{equation}
\xi(x):=\nabla\log\det F(x)=(F_i\bullet F(x)^{-1})_{i=1}^n.\label{eq:wx}
\end{equation}
Solving the above problem is still difficult since it contains the semi-infinite constraints
$
g(x,\tau)+\nabla_xg(x,\tau)^{\top}\Delta x\le 0\ (\tau\in T).
$
To relax the difficulty, we propose to make use of its inexact solution 
$\Delta x\in \R^n$ together with Lagrange multiplier measure $y^+\in \mathcal{M}_+(T)$ satisfying $\left|{\rm supp}(y^+)\right|<\infty$ and vector $z^+\in \R^s$ 
such that 
\begin{align}\label{al:opt}
&\nabla f(x)+B\Delta x-\mu \xi(x)+\int_T\nabla_xg(x,\tau)\ydtp+G^{\top}z^+=0,\notag \\
&g(x,\tau)+\nabla_xg(x,\tau)^{\top}\Delta x\le 0\ \ (\tau\in {\rm supp}(y^+)),\\
&\int_T\left(g(x,\tau)+\nabla_xg(x,\tau)^{\top}\Delta x\right)\ydtp=0,\ G(x+\Delta x)=h,\notag \\
&\max_{\tau\in T}\left(
g(x,\tau)+\nabla_xg(x,\tau)^{\top}\Delta x
\right)_+\le \gamma,\notag
\end{align}
where $\gamma>0$ is a relaxation parameter controlled in the algorithm.
Thanks to $|{\rm supp}(y^+)|<\infty$, the above integral forms can be replaced with simple finite summations:
\begin{align*}
&\int_T\nabla_xg(x,\tau)\ydtp=\sum_{j=1}^p\nabla_xg(x,\tau_j)y^+({\tau_j}),\\
&\int_T\left(g(x,\tau)+\nabla_xg(x,\tau)^{\top}\Delta x\right)\ydtp=
\sum_{j=1}^p\left(g(x,\tau_j)+\nabla_xg(x,\tau_j)^{\top}\Delta x\right)y^+({\tau_j})
\end{align*}
with $p:=|{\rm supp}(y^+)|$ and ${\rm supp}(y^+)=\{\tau_1,\tau_2,\ldots,\tau_p\}$.

Notice here that, if $\gamma=0$, then 
the conditions\,\eqref{al:opt} are noting but the KKT conditions for SIQP\,\eqref{SIQP}. 
We should further remark that exchange-type methods \cite{Reem,Lai, okuno2012regularized,okuno2016exchange} work effectively in finding vectors satisfying those conditions. 
Below, an exchange method for finding $\left(\Delta x,y^+,z^+\right)$ satisfying the conditions \eqref{al:opt} is described:
\begin{center}
{\underline{Exchange method}}
\end{center}
\begin{description}
\item{Step~0:} Choose the initial {index} set $T_0\subsetneq T$ such that $|T_0|<\infty$. Set $k:=0$.
\item{Step~1:} Solve SIQP\,\eqref{SIQP} with $T$ replaced by $T_k$ to obtain an optimum $\Delta x^k$ and Lagrange multipliers $\zeta_{\tau}\ge 0\ (\tau\in T_k)$ corresponding to the inequality constraints. 
\item{Step~2:} Set $\tilde{T}_{k}:=\{\tau\in T_k\mid \zeta_{\tau}>0\}$.
\item{Step~3:} Find an index $\tau\in T$ such that $g(x,\tau)+\nabla g(x,\tau)^{\top}\Delta x>\gamma$ and let $T_{k+1}:=\tilde{T}_k\cup \{\tau\}$.
If such an index $\tau\in T$ does not exist, i.e., $\max_{\tau\in T}\left(g(x,\tau)+\nabla g(x,\tau)^{\top}\Delta x\right)\le \gamma$ holds, stop the algorithm. Otherwise, go to Step~4.
\item{Step~4:} Set $k:=k+1$ and return to Step~1.
\end{description}
In Step~2, we drop indices corresponding to the inequality constraints with zero Lagrange multipliers, which contain inactive constraints. 
Particularly,
it can be proved in a manner similar to \cite[Theorem~3.2]{okuno2012regularized}
that 
under the positive-definiteness of the matrix $B$, 
the above exchange method stops in finitely many iterations.

We next consider how we derive $\Delta V$ by means of scaling techniques. 
Choose a nonsingular matrix $P\in \R^{m\times m}$ arbitrarily and scale $F(x)$ and $V$ by 
\begin{align}
&\tilde{F}(x):=PF(x)P^{\top},\ \tilde{V}:=P^{-\top}VP^{-1}.\label{scal}                                                                                   
\end{align}
Note that the barrier shifted complementarity conditions
$F(x)\circ V=\mu I,\ F(x)\in S^m_+,\ V\in S^m_+$ are equivalent to the scaled ones  
$\tilde{F}(x)\circ \tilde{V}=\mu I,\ \tilde{F}(x)\in S^m_+,\ \tilde{V}\in S^m_+$.
Then, the Newton equations for $\tilde{F}(x)\circ \tilde{V}=\mu I$ are written as
\begin{equation}
\left(\tilde{F}(x)+\sum_{i=1}^n\Delta x_i\tilde{F}_i\right)\circ\tilde{V}+
\tilde{F}(x)\circ\Delta \tilde{V}=\mu I. \label{eq:scaled_new}
\end{equation}
Here, $\tilde{F}_i:=PF_iP^{\top}$ for $i=1,2,\ldots,n$.
In terms of the linear operator $\mathcal{L}_{\tilde{F}(x)}:S^m\to S^m$,
\eqref{eq:scaled_new}
is rephrased as 
\begin{equation}
\sum_{i=1}^n\Delta x_i{\tilde{F}_i}\circ\tilde{V}+
\mathcal{L}_{\tilde{F}(x)}\Delta \tilde{V}=\mu I-\mathcal{L}_{\tilde{F}(x)}\tilde{V}.\label{eq:scaled_new_liap}
\end{equation}
Under the condition that ${F}(x)\in S^m_{++}$, 
$\mathcal{L}_{\tilde{F}(x)}$ is invertible, and hence \eqref{eq:scaled_new_liap} is uniquely solvable for $\Delta\tilde{V}$. Actually, we have 
\begin{equation}
\Delta \tilde{V}=\mu\tilde{F}(x)^{-1}-\sum_{i=1}^n\Delta x_i\mathcal{L}_{\tilde{F}(x)}^{-1}\mathcal{L}_{\tilde{V}}\tilde{F}_i-\tilde{V}.\label{eq:0326}
\end{equation}
Now, we propose to derive $\Delta V$ by the inverse-scaling of $\Delta \tilde{V}$. Specifically, 
$\Delta V$ is computed {as}
\begin{equation}
\Delta V=P^{\top}\Delta \tilde{V}P=
\mu{F}(x)^{-1}-V-\sum_{i=1}^n\Delta x_iP^{\top}\mathcal{L}_{\tilde{F}(x)}^{-1}\mathcal{L}_{\tilde{V}}\tilde{F}_iP.
\label{eq:dv}
\end{equation} 
The direction $\Delta V$ obtained by \eqref{eq:dv} may be seen as a member of the family of Monteiro-Zhang (MZ) directions\,\cite[Chapter~10]{wolkowicz2012handbook}.
Depending on the choice of $P$, generated directions admit different properties.
In particular, the following selections of $P$ and the correspondingly obtained directions are significant in the context of the LSDPs and NSDPs.
\begin{description}
\item[AHO direction ($P=I$):]
$\Delta V=\mu{F}(x)^{-1}-V-\sum_{i=1}^n\Delta x_i\mathcal{L}_{{F}(x)}^{-1}({F}_i\circ{V}).$  
\item[HRVW/KSH/M direction ($P=F(x)^{-\frac{1}{2}}$ ):]
In this case, $\tilde{F}(x)=I$ and $\Delta V=\mu{F}(x)^{-1}-V-\left(F(x)^{-1}\left(\sum_{i=1}^n\Delta x_iF_i\right)V+V\left(\sum_{i=1}^n\Delta x_iF_i\right)F(x)^{-1}\right)/2$.
\item[NT direction ($P=W^{-\frac{1}{2}},\ W:=F(x)^{\frac{1}{2}}(F(x)^{\frac{1}{2}}VF(x)^{\frac{1}{2}})^{-\frac{1}{2}}F(x)^{\frac{1}{2}})$:]
In this case,  $\tilde{F}(x)=\tilde{V}$ and
$\Delta {V}=\mu{F}(x)^{-1}-V-W^{-1}\left(\sum_{i=1}^n\Delta x_i{F}_i\right)W^{-1}$.
\end{description}
As for the HRVW/KSH/M and NT directions, 
we should note that $\tilde{F}(x)$ and $\tilde{V}$ commute, namely, $\tilde{F}(x)\tilde{V}=\tilde{V}\tilde{F}(x)$.
Hereafter, we call the scaling matrices for making AHO, HRVW/KSH/M, and NT directions 
AHO, HRVW/KSH/M, and NT matrices, respectively. 
\subsection{Step size along $(\Delta x,\Delta V)$}
To find a step size $s\in (0,1]$ along the obtained search direction $\Delta W:=(\Delta x, \Delta V)$, we use an Armijo-like line search technique along with the merit function $\Phi_{\rho}:\R^n\times S^m\to \R$ defined {by} 
\begin{equation}
\Phi_{\rho}(x,V): = \chi_{\rho}(x)+\nu\psi(x,V),\label{merit}
\end{equation}
where $\nu>0$ is a positive parameter, $\rho>0$ is a penalty parameter, and  
\begin{align}
\chi_{\rho}(x)&:= f(x) - \mu\log\det F(x)+\rho\max_{\tau\in T}\left(g(x,\tau)\right)_+ +\rho\|Gx-h\|_{1}, \label{al:chi}\\
\psi(x,V)&:=F(x)\bullet V-\mu\log\det F(x)V.\notag 
\end{align}
The function $\Phi_{\rho}$ is a straightforward extension of the primal-dual barrier merit function for getting the BKKT point of the NSDP\,\cite{yabe}.
The function $\psi(\cdot,\cdot)$ can be regarded as a merit function for the barrier shifted complementarity condition
$
F(x)\circ V=\mu I, F(x)\in S^m_+,\ V\in S^m_+.
$
Actually, when $F(x)\in S^m_+$ and $V\in S^m_+$, it holds that 
\begin{equation}
\nabla \psi(x,V)=0
\Longleftrightarrow 
F(x)\circ V=\mu I. 
\end{equation}

A step size $s>0$ is then determined using the Armijo-like line search method, so that the value of $\Phi_{\rho}$ approximately decreases while keeping $F(x)\in S^m_{++}$ and $V\in S^m_{++}$:
More specifically, we first choose parameters $\alpha$, $\beta\in (0,1)$ arbitrarily.
Then, we find the smallest nonnegative integer $\ell$ such that
\begin{align*}
\Phi_{\rho}(x+\bar{s}\beta^{\ell}\dx,V+\bar{s}\beta^{\ell}\dV)&\\
&\hspace{-5em}\le \Phi_{\rho}(x,V)-\alpha \bar{s}\beta^{\ell}\left(\dx^{\top}B\dx+\nu {\psi}^{\prime}(x,V;\dx,\dV)\right)+\bar{s}\beta^{\ell}\rho\gamma,
\end{align*}
where 
$B\in S^n_{++}$ is the matrix prescribed in SIQP\,\eqref{SIQP}, 
$\psi^{\prime}(x,V;\dx,\dV)$ denotes the directional derivative of function $\psi$ at $(x,V)$ in the direction $(\Delta x,\Delta V)$ and it is explicitly represented as 
\begin{equation}\label{direct}
\psi^{\prime}(x,V;\Delta x,\Delta V)={\rm Tr}\left(\sum_{i=1}^n\Delta x_iF_iV+F(x)\Delta V-\mu F(x)^{-1}\sum_{i=1}^n\Delta x_iF_i-\mu V^{-1}\Delta V\right).
\end{equation}
Also, $\gamma>0$ is the constant prescribed in \eqref{al:opt}. 
Moreover, $\bar{s}\in (0,1]$ is the initial step size chosen so that
\begin{equation}
  F(x)+s\sum_{i=1}^n\Delta x_iF_i\in S^m_{++}\ \mbox{and }V+s\Delta V\in S^m_{++}\notag 
\end{equation}
hold for any $s\in [0,\bar{s}]$. 
For example, we set
\begin{equation}
\bar{s}=\min(s_x,s_V,1),\label{eq:s0}
\end{equation}
where
\begin{align*}
s_x&:=
\begin{cases}
\displaystyle{- \frac{\sigma}{\lambda_{\rm min}(F(x)^{-1}\sum_{i=1}^n\dx_iF_i)}} &\mbox{if }\lambda_{\rm min}(F(x)^{-1}\sum_{i=1}^n\dx_iF_i)<0\\
1                                                     &\mbox{otherwise}, 
\end{cases}\\
s_V&:=
\begin{cases}
\displaystyle{- \frac{\sigma}{\lambda_{\rm min}(V^{-1}\Delta V)}} &\mbox{if }\lambda_{\rm min}(V^{-1}\Delta V)<0\\
1                                                     &\mbox{otherwise} 
\end{cases}
\end{align*}
with $\sigma\in (0,1)$ a positive parameter.
Furthermore, to ensure that the generated direction $\left(\dx,\dV\right)$ is a decent direction for $\Phi_{\rho}$,
the value of the penalty parameter $\rho$ must be chosen sufficiently large.
Specifically, 
we set $\rho$ so that  
\begin{equation*}
\rho>\max\left(\|y\|,\|z\|_{\infty}\right).
\end{equation*}

Now, we describe the algorithm for getting BKKT points.\vspace{1.0em}
\begin{center}
\underline{Algorithm~1}\vspace{1.0em}
\end{center}
   \begin{description}
\item[Step~0 (Initialization):]
Set $(\ru[x][0],\ru[z][0],\rl[V][0])\in\R^n\times \R^s\times S^m_{++}$ and 
$\ru[y][0]\in \mathcal{M}_+(T)$ such that $|{\rm supp}(y^0)|<\infty$.
Choose parameters 
$
\alpha,\beta_1,\beta_2,\sigma\in (0,1)$, and $\delta, \nu, \rl[\rho][0], \rl[\gamma][0]>0.$ 
Set $r:=0$.
\item[Step~1 (Stopping condition):]
If 
{$(x^r,y^r,z^r,V_r)$
satisfies the KKT conditions\,\eqref{e1}--\eqref{e4}}, 
then
stop the algorithm. Otherwise, go to Step~2.
\item[Step~2 (Select $P_{r}$ and $B_{r}$):] 
Choose a nonsingular matrix $\rl[P][r]\in \R^{n\times n}$ and 
positive definite matrix $\rl[B][r]\in S^n$.
\item[Step~3 (Generate $(\Delta x^{r}$, $\Delta V_{r}, y^{r+1},z^{r+1})$):]
Find a search direction $\rl[\dx][r]$,
Lagrange multiplier {measure} $\ru[y][r+1]\in\mathcal{M}_+(T)$
{such that $|{\rm supp}(y^{r+1})|<\infty$}, 
and vector $\ru[z][r+1]\in\R^s$ satisfying \eqref{al:opt} with
$x=x^r$, $B=B_r$, and $\gamma=\gamma_r$.
Compute $\rl[\Delta V][r]$ from \eqref{eq:dv} with $x=x^r$.
\item[Step~4 (Update $\rho_r$):] If
\begin{equation*}
\rl[\rho][r]>\max\left(\|\ru[y][r+1]\|,\ \|\ru[z][r+1]\|_{\infty}\right),
\end{equation*}
set $\rl[\rho][r+1]:=\rl[\rho][r]$. Otherwise, set
\begin{equation}
\rl[\rho][r+1]:=\delta+
\max\left(\|\ru[y][r+1]\|,\ \|\ru[z][r+1]\|_{\infty}\right).
\end{equation}
\item[Step~5 (Armijo-like line search):]
Compute $\bar{s}$ by \eqref{eq:s0} and let $\rl[s][r]=\bar{s}\beta_1^{\ell}$ with
the smallest nonnegative integer $\ell\ge 0$ satisfying   
\begin{equation}\label{eq:armijo-rule}
\Phi_{\rl[\rho][r+1]}(\ru[x][r]+\rl[s][r]\ru[\dx][r],V_r+\rl[s][r]\rl[\dV][r])\le \Phi_{\rl[\rho][r+1]}(\ru[x][r],\rl[V][r])-\alpha \rl[s][r]\Delta\Phi_r\ +\rl[\rho][r+1]s_r\rl[\gamma][r],
\end{equation}
where $\Delta\Phi_r:=(\ru[\dx][r])^{\top}\rl[B][r]\ru[\dx][r]-\nu \psi^{\prime}(\ru[x][r],\rl[V][r];\ru[\dx][r],\rl[\dV][r])$.
\item[Step~6 (Update $x^{r}$, $V_{r}$ and $\gamma_{r}$):] Set
\begin{equation*}
(\ru[x][r+1],\rl[V][r+1]):=(\ru[x][r]+\rl[s][r]\ru[\dx][r],\rl[V][r]+\rl[s][r]\rl[\dV][r]),
\end{equation*}
$\gamma_{r+1}:=\beta_2\gamma_{r}$, and $r:=r+1$.
Return to Step~1. 
\end{description}
\subsection{Choice of the coefficient matrix $B_r$}\label{subsec:Br}
In this section, we consider a particular choice of the coefficient matrix $B_r$ in SIQP\,\eqref{SIQP}.
In the conventional SQP, one of the basic selections for $B_r$ is the Hessian of the Lagrangian for
the SIPLOG\,\eqref{lsisdp}, i.e., 
\begin{equation}
\nabla_{xx}^2L(x^r,y^r,z^r)=\nabla^2f(x^r)+\int_T\nabla^2_{xx}g(x^r,\tau)dy^r(\tau)-\mu\nabla^2\log\det {F(x^r),}\label{eq:hessian}
\end{equation}
where  
$L(x,y,z):=f(x)+\int_Tg(x,\tau)dy(\tau)+(Gx-h)^{\top}z-\mu\log\det F(x)$.
To explore other candidates for $B_r$,
we consider 
the matrix valued function $B:\R^n\times \mathcal{M}(T)\times S^m_{++}\to S^n_{++}$ of the following form: 
\begin{equation}
B(x,y,V):={M}(x,y)+{H}_{P}(x,V),\label{eq:0913}
\end{equation}
where ${M}: \R^n\times \mathcal{M}(T)\to S^n$ 
is some positive definite matrix valued function and ${H}_P:\R^n\times S^m_{++}\to S^n$ 
is defined by  
\begin{equation}
\left({H}_P(x,V)\right)_{i,j}:=\frac{\tilde{F}_i\bullet\left(\mathcal{L}_{\tilde{F}(x)}^{-1}\mathcal{L}_{\tilde{V}}
+\mathcal{L}_{\tilde{V}}\mathcal{L}_{\tilde{F}(x)}^{-1}
\right)\tilde{F}_j}{2}\label{eq:HP}
\end{equation}
for $i,j=1,2,\ldots,n$ and $(x,V)\in \R^n\times S^m$. Recall here that $\tilde{F}(x)$ and $\tilde{V}$ are positive definite matrices
obtained by scaling $F(x)$ and $V$ with the matrix $P$. See \eqref{scal}.
When $\tilde{F}(x)$ and $\tilde{V}$ commute,  
{so} do $\mathcal{L}^{-1}_{\tilde{F}(x)}$ and 
$\mathcal{L}_{\tilde{V}}$.
Then, 
by noting \eqref{eq:0326},
the KKT conditions\,\eqref{al:opt} 
of SIQP\,\eqref{SIQP} with $B=B(x,y,V)$ can be represented in terms of $\Delta\tilde{V}$ as  
\begin{align*}\label{al:opt2}
&M(x,y)\Delta x+\nabla f(x)-
                                                         (\tilde{F}_i\bullet (\tilde{V}+\Delta\tilde{V}))_{i=1}^n 
+\int_T\nabla_xg(x,\tau)d(y+\Delta y)(\tau)+G^{\top}(z+\Delta z)=0,\notag \\
&\tilde{F}(x)\circ \tilde{V}+\tilde{F}(x)\circ \Delta \tilde{V}
+\sum_{i=1}^n\Delta x_i\tilde{F}_i\circ\tilde{V} =\mu I,\notag\\
&g(x,\tau)+\nabla_xg(x,\tau)^{\top}\Delta x\le 0\ \ (\tau\in {\rm supp}(y+\Delta y)),\\
&\int_T\left(g(x,\tau)+\nabla_xg(x,\tau)^{\top}\Delta x\right)d(y+\Delta y)(\tau)=0,\ y+\Delta y\in\mathcal{M}_+(T),\ G(x+\Delta x)=h,\notag \\
&\max_{\tau\in T}\left(
g(x,\tau)+\nabla_xg(x,\tau)^{\top}\Delta x
\right)_+\le 0.\notag
\end{align*}
Actually, by substituting $\tilde{V}+\Delta\tilde{V}=\mathcal{L}_{\tilde{F}(x)}^{-1}(\mu I - \sum_{i=1}^n\Delta x_i\tilde{F}_i\circ\tilde{V})$, which is obtained from the above second condition, into the first condition, we can regain the KKT conditions\,\eqref{al:opt} with $B=B(x,y,V)$.
Furthermore, if the function $M(x,y)$ is the Hessian of the function 
\begin{equation}
L_2(x,y):=f(x)+\int_Tg(x,\tau)dy(\tau),\label{eq:0615-0916}
\end{equation}
namely, 
$\nabla^2_{xx}L_2(x,y) = \nabla^2f(x)+\int_T\nabla^2_{xx}g(x,\tau)dy(\tau)$, then solving the above system is regarded as the scaled Newton iteration for the KKT conditions.
From these observations, we can expect that employing such $B(x^r,y^r,V_r)$ as $B_r$ accomplishes a rapid convergence.
{For reference, we list the formulas of $({H}_P(x,V))_{ij}$
below for the case where the HRVW/KSH/M and NT matrices are selected as the scaling matrix $P$.
\begin{description}
\item[HRVW/KSH/M matrix:] $({H}_P(x,V))_{ij}={\rm Tr}\left(F(x)^{-1}F_jVF_i\right)$,
\item[NT matrix:] $({H}_P(x,V))_{ij}={\rm Tr}\left(W^{-1}F_iW^{-1}F_j\right)$ with $W=F(x)^{\frac{1}{2}}(F(x)^{\frac{1}{2}}VF(x)^{\frac{1}{2}})^{-\frac{1}{2}}F(x)^{\frac{1}{2}}$.
\end{description}}
\subsection{Convergence analysis of Algorithm~1}
In the subsequent analysis, 
we make the following assumptions:
\begin{description}
\item{\bf Assumption~A:}
\end{description}
\begin{enumerate}
\item\label{B2}
The scaling matrix $P_r$ is the HRVW/KSH/M or NT matrix for any $r=0,1,2,\ldots$.
\item\label{B3} The sequence of penalty parameters $\{\rho_{r}\}$ is bounded.
\item\label{B4} The generated sequence $\{x^{r}\}$ is bounded.
\item\label{B5} The initial point $\ru[x][0]$ is chosen to satisfy $G\ru[x][0]=h$. 
\end{enumerate}
Although similar assumptions are made in many existing works on interior point methods for nonlinear programs or NSDPs,
Assumption~A-\ref{B3} may seem rather strong{.}
To relax it is one of future subjects that should be settled.
As for Assumption~A-\ref{B4}, we can show its validity under a certain hypothesis; see Proposition~\ref{bound_x} below.
Assumption A-\ref{B5} is made for simplicity of expression in some subsequent proofs. 
The proofs can be extended straightforwardly to the general case {where} Assumption~A-\ref{B5} {is absent}. 
\begin{Prop}\label{bound_x}
Suppose that Assumption~A-\ref{B5} hold{s} and $g(x,\tau)$
is an affine function with respect to $x$, i.e., $g(x,\tau)=a(\tau)^{\top}x-b(\tau)$,
where $a:T\to\R^n$ and $b:T\to \R$ are continuous functions.
Furthermore, assume that the feasible set of the SISDP\,\eqref{lsisdp} is nonempty and compact.
Then, the generated sequence $\{\ru[x][r]\}$ is bounded.
\end{Prop}
\begin{proof}
We can show that for each $r\ge 0$
\begin{equation*}
F(x^{r})\in S^m_{++},\ a(\tau)^{\top}x^{r}-b(\tau)\le \delta_0\ \ (\tau\in T),\ Gx^{r}=h,
\end{equation*}
where $\delta_0:=\max\left(\max_{\tau\in T}\left(a(\tau)^{\top}x_0-b(\tau)\right),\gamma_0\right)$.
We prove only the second relation by mathematical induction. 
It holds for $r=0$ obviously. Next, suppose that it holds true for some $r>0$. Then, by noting $\max_{\tau\in T}\left(a(\tau)^{\top}(x^r+\Delta x^r)-b(\tau)\right)\le \gamma_r\le \gamma_0$ and $s_r\in (0,1]$, we have  
\begin{align*}
a(\tau)^{\top}x^{r+1}-b(\tau)&=a(\tau)^{\top}(x^r+s_r\Delta x^r)-b(\tau)\\
                                    &\le s_r\gamma_0+(1-s_r)(a(\tau)^{\top}x^r-b(\tau))\\
                                    &\le \max\left(a(\tau)^{\top}x^r-b(\tau),\gamma_0\right)\\
                                    &\le  \delta_0
\end{align*}
for any $\tau\in T$. Therefore, we ensure the targeted inequality holds for all $r\ge 0$.

Denote the feasible set of the SISDP\,\eqref{lsisdp} by $\mathcal{F}$ and 
define a proper closed convex function ${\pi}:\R^n\to \R$ by
$$
{\pi}(x):={\max}\left(-\lambda_{\min}(F(x)),\ \max_{\tau\in T}a(\tau)^{\top}x-b(\tau),\ \|Gx-h\|\right).
$$
Since the level set $\{x\in \R^n \mid {\pi}(x)\le 0\}(=\mathcal{F})$ is compact from the assumption that $\mathcal{F}$ is nonempty and compact, any level set $\{x\in \R^n \mid {\pi}(x)\le \eta\}$ with $\eta>0$ is also compact. Then, 
we can see that 
$\{x^r\}\subseteq \{x\in \R^n \mid {\pi}(x)\le \delta_0\}$ and thus $\{x^k\}$ is bounded.
\end{proof}

Now, we enter the essential part of the global convergence of the algorithm.

The following lemmas play key roles in establishing the well-definedness of the Armijo-like linesearch in Step~5.  
\begin{Lem}\label{lem:0819}
For any $x\in \R^n$ with $F(x)\in S^m_{++}$ and any $V\in S^m_{++}$, it holds that
\begin{equation*}
\psi(x,V)\ge m\mu (1 -\log \mu),
\end{equation*}
where the equality holds if and only if $F(x)V=\mu I$.
\end{Lem}
\begin{proof}
Denote the eigenvalues of $F(x)V$ by $\lambda_i>0\ (i=1,2,\ldots,m)$. 
Then, 
$\psi(x,V)={\rm Tr}(F(x)V)-\mu\log\det F(x)V=\sum_{i=1}^m(\lambda_i-\mu\log\lambda_i)\ge m\mu(1-\log\mu)$.
The equality holds if and only if $\lambda_1=\lambda_2=\cdots=\lambda_m=\mu$, that is, $F(x)V=\mu I$.
\end{proof}

\begin{Lem}\label{lem_bp}
\begin{enumerate}
\item\label{lem_bp1} 
It holds that 
\begin{equation}
\psi^{\prime}(x^{r},V_{r};\Delta x^{r},\Delta V_{r})\le 0.\label{eq:0326-1}
\end{equation}
In particular, the equality holds if and only if $F(x^r)V_r=\mu I$. \vspace{0.5em}\\
\item 
Let
$\theta(x):=\max_{\tau\in T}\left(g(x,\tau)\right)_+$
and $\theta^{\prime}(x;\Delta x)$ be
the directional derivative of $\theta$ at $x$ in the direction $\dx$.
Then, $\theta(\ru[x][r])+\theta^{\prime}(\ru[x][r];\ru[\dx][r])\le \rl[\gamma][r]$ holds.
\end{enumerate}
\end{Lem}
\begin{proof}
{\rm 1.} Although the proof can be given in a fashion similar to \cite[Lemma~3]{yabe}, we show it here for completeness. 
We first prove $\psi^{\prime}(\ru[x][r],\rl[V][r];\ru[\dx][r],\rl[\dV][r])\le 0$.
Note first that 
\begin{equation}
\mu{\rm Tr}(\tilde{F}(x^r)^{-1}\sum_{i=1}^n\Delta x_i^r\tilde{F}_i+\tilde{V}_r^{-1}\Delta\tilde{V}_r)={\rm Tr}(\mu^2\tilde{F}(x^r)^{-1}\tilde{V}_r^{-1}-\mu I),\label{eq:0601}
\end{equation}
which is implied by the scaled Newton equation $\tilde{F}(\ru[x][r])\circ\rl[\tilde{V}][r]+\sum_{i=1}^n\ru[\dx_i][r]\tilde{F}_i\circ\rl[\tilde{V}][r]+\tilde{F}(\ru[x][r])\circ\rl[\tilde{\dV}][r]=\mu I$ together with $\tilde{F}(x^r)\tilde{V}_r=\tilde{V_r}\tilde{F}(x^r)$.\\
Then, using $\tilde{F}(x^r)\tilde{V}_r=\tilde{V}_r\tilde{F}(x^r)$ again, 
we have from \eqref{direct} 
\begin{align}
&\hspace{1em}\psi^{\prime}(\ru[x][r],\rl[V][r];\ru[\dx][r],\rl[\dV][r])\notag \\
&={\rm Tr}\left(\sum_{i=1}^n\ru[\dx_i][r]F_i\rl[V][r]+F(\ru[x][r])\rl[\dV][r]-\mu F(\ru[x][r])^{-1}\sum_{i=1}^n\ru[\dx_i][r]F_i-\mu \rl[V][r]^{-1}\rl[\dV][r]\right)\notag \\
&={\rm Tr}\left(\sum_{i=1}^n\ru[\dx_i][r]\tilde{F}_i\rl[\tilde{V}][r]+\tilde{F}(\ru[x][r])\rl[\Delta\tilde{V}][r]-\mu \tilde{F}(\ru[x][r])^{-1}\sum_{i=1}^n\ru[\dx_i][r]\tilde{F}_i-\mu \rl[\tilde{V}][r]^{-1}\Delta \tilde{V}_r\right)\notag \\
&={\rm Tr}\left(\mu I-\tilde{F}(\ru[x][r])\rl[\tilde{V}][r]-\mu \tilde{F}(\ru[x][r])^{-1}\sum_{i=1}^n\ru[\dx_i][r]\tilde{F}_i-\mu \rl[\tilde{V}][r]^{-1}\rl[\Delta\tilde{V}][r]\right)\notag\\
&={\rm Tr}\left(2\mu I-\tilde{F}(\ru[x][r])\rl[\tilde{V}][r]-\mu^2\rl[\tilde{V}][r]^{-1}\tilde{F}(x^r)^{-1}\right)\notag\\
&=-\left\|\mu\rl[\tilde{V}][r]^{-\frac{1}{2}}\tilde{F}(\ru[x][r])^{-\frac{1}{2}}-\tilde{F}(\ru[x][r])^{\frac{1}{2}}\rl[\tilde{V}][r]^{\frac{1}{2}}\right\|_F^2\label{al:0819-direct}\\
&\le 0,
\end{align}
where the third equality holds because  
$$
{\rm Tr}\left(
\mu I -\tilde{F}(x^r)\tilde{V}_r
\right)
={\rm Tr}\left(
\sum_{i=1}^n\Delta x_i^r\tilde{F}_i\tilde{V}_r+\Delta\tilde{V}_r\tilde{F}(x^r)\right)
$$
from the scaled Newton equations, 
the fourth equality follows from \eqref{eq:0601}.
Thus, we get \eqref{eq:0326-1}.
By \eqref{al:0819-direct}, we observe that
$\psi^{\prime}(x,V;x^r,V_r)=0$ if and only if
$
\mu\rl[\tilde{V}][r]^{-\frac{1}{2}}\tilde{F}(\ru[x][r])^{-\frac{1}{2}}-\tilde{F}(\ru[x][r])^{\frac{1}{2}}\rl[\tilde{V}][r]^{\frac{1}{2}}=O,
$
whch is equivalent to $\tilde{F}(x^r)\tilde{V}_r=\mu I$, i.e., 
${F}(x^r){V}_r=\mu I$.
We therefore obtain the latter claim. \vspace{0.5em}\\
{\rm 2.} Let $T(x):={\rm argmax}_{\tau\in T}g(x,\tau)$.
Consider the three cases where the value of $\max_{\tau\in T}g(\ru[x][r],\tau)$ is (i) $<0$, (ii) $>0$, and (iii) {$=0$}:
\begin{enumerate}
\item[(i)] In this case, $\theta(\ru[x][r])=0$ and $\theta^{\prime}(\ru[x][r];\ru[\dx][r])=0$ holds.
Then $\theta(\ru[x][r])+\theta^{\prime}(\ru[x][r];\ru[\dx][r])=0<\rl[\gamma][r]$ readily follows. 
\item[(ii)] In this case, we have
$\theta(\ru[x][r])=\max_{\tau\in T(\ru[x][r])}g(\ru[x][r],\tau)>0$. 
We then get $\theta^{\prime}(\ru[x][r];\ru[\dx][r])=\max_{\tau\in T(\ru[x][r])}\nabla_xg(\ru[x][r],\tau)^{\top}\ru[\dx][r]$ and therefore
\begin{align*}
\theta(\ru[x][r])+\theta^{\prime}(\ru[x][r];\ru[\dx][r])
&=\max_{\tau\in T(\ru[x][r])}\left(\theta(x^r)+\nabla_xg(\ru[x][r],\tau)^{\top}\ru[\dx][r]\right)\\
&=\max_{\tau\in T(\ru[x][r])}
\left(g(\ru[x][r],\tau)+\nabla_xg(\ru[x][r],\tau)^{\top}\ru[\dx][r]\right)\le\rl[\gamma][r].
\end{align*}
\item[(iii)] In this case, we have $\theta(\ru[x][r])=g(x^r,\tau)=0$
for each $\tau\in T(x^r)$ and $\theta^{\prime}(\ru[x][r];\ru[\dx][r])=\max_{\tau\in T(\ru[x][r])}\left(\nabla_xg(\ru[x][r],\tau)^{\top}\ru[\dx][r]\right)_+$.
Then, 
$$
\theta(\ru[x][r])+\theta^{\prime}(\ru[x][r];\ru[\dx][r])=\theta^{\prime}(\ru[x][r];\ru[\dx][r])=\max_{\tau\in T(\ru[x][r])}\left(g(\ru[x][r],\tau)+\nabla_xg(\ru[x][r],\tau)^{\top}\ru[\dx][r]\right)_+\le\rl[\gamma][r].
$$
\end{enumerate}
We have obtained the desired conclusion.
\end{proof}
\begin{Lem}\label{lem:0820}
Suppose $\Delta x=0$. Then, ${\psi^{\prime}}(x,V;\Delta x,\Delta V)=0$ implies $\Delta V=O$.
\end{Lem}
\begin{proof}
From Lemma\,\ref{lem_bp}\eqref{lem_bp1} and ${\psi^{\prime}}(x,V;\Delta x,\Delta V)=0$, 
we have $F(x)V=\mu I$, i.e., $\tilde{F}(x)\tilde{V}=\mu I$. 
This together with $\Delta x=0$ implies that the scaled Newton equations\,\eqref{eq:scaled_new}
yield $\tilde{F}(x)\circ \Delta \tilde{V}=O$, i.e., $\mathcal{L}_{\tilde{F}(x)}\Delta\tilde{V}=O$.
It then follows that $\Delta\tilde{V}=O$ since 
$\mathcal{L}_{\tilde{F}(x)}$ is invertible by $\tilde{F}(x)\in S^m_{++}$. Consequently, we have $\Delta V=O$.
\end{proof}
\begin{Prop}\label{Prop:0329}
For any sufficiently small $s\in (0,1]$, 
we have 
\begin{equation}
\Phi_{\rho_r}(\ru[x][r]+s\ru[\dx][r],\rl[V][r]+s\rl[\dV][r])\le \Phi_{\rho_r}(
\ru[x][r],\rl[V][r])-s\alpha (\Delta \ru[x][r])^{\top}B_{r}\Delta \ru[x][r]+\Delta\psi+
s\rho_r \rl[\gamma][r],
\end{equation}
where
$
\Delta\psi:=\psi^{\prime}(\ru[x][r],\rl[V][r];\Delta \ru[x][r],\Delta \rl[V][r]).
$
\end{Prop} 
\begin{proof}
If $\left(\ru[\dx][r],\rl[\dV][r]\right)=(0,O)$, then $\psi^{\prime}(x^r,V_r;\Delta x^r,\Delta V_r)=0$ holds and the desired conclusion obviously holds for any $s>0$. So, we provide with the proof
by assuming $\left(\ru[\dx][r],\rl[\dV][r]\right)\neq (0,O)$.

For simplicity of expression, we abbreviate $\rho_r$, $x^{r}$, $\Delta x^{r}$, and $B_{r}$ as $\rho$, $x$, $\Delta x$, and $B$, respectively. 
In addition, we write ${\rm supp}(y)=\{\tau_1,\tau_2,\ldots,\tau_p\}$ and represent $y(\tau_i)$ as $y_i$ for each $i$.
Note that
$\|y\|=\sum_{i=1}^py_i$ by $y\in \mathcal{M}_+(T)$.

Let $f_{\rm bp}(x):=f(x)-\mu\log\det F(x)$ and 
$\theta(x):=\max_{\tau\in T}\left(g(x,\tau)\right)_+$.
Moreover, define $\Delta\psi:=\psi^{\prime}(x,V;\Delta x,\Delta V)$, $w:=(x,V)$, and $\Delta w:=(\dx,\dV)$.
We then have
{\begin{align}
\Phi_{\rho}(w+s\Delta w)-\Phi_{\rho}(w)&=s\nabla f_{\rm bp}(x)^{\top}\Delta x +s\nu\Delta\psi+
s\rho\theta^{\prime}(x;\dx)+o(s)\notag \\
&=-s\Delta x^{\top}B\Delta x
+
s\sum_{i=1}^pg(x,\tau_i)y_i
+s\rho\theta^{\prime}(x;\dx)+s\nu\Delta\psi+o(s)\notag \\
&\le -s\Delta x^{\top}B\Delta x+s\left(\sum_{i=1}^py_i\right)\theta(x)+
s\rho\theta^{\prime}(x;\dx)+s\nu\Delta\psi+o(s)\notag\\
&\le -s\Delta x^{\top}B\Delta x+s\rho(\theta(x)+\theta^{\prime}(x;\dx))+s\nu\Delta\psi+o(s)\notag\\
&\le -s\Delta x^{\top}B\Delta x+s\rho\gamma+s\nu\Delta\psi+o(s),\label{a3}
\end{align}}
where the second equality follows from the KKT conditions
\begin{align}
&B\Delta x+\nabla f_{\rm bp}(x)+\sum_{i=1}^p\nabla_xg(x,\tau_i)y_i+G^{\top}z=0,\ G\Delta x = 0,\notag\\
&0\ge g(x,\tau_i)+\nabla_xg(x,\tau_i)^{\top}\Delta x,\
y_i\ge 0,\ \left(g(x,\tau_i)+\nabla_xg(x,\tau_i)^{\top}\Delta x\right)y_i=0\ (i=1,2,\ldots,p),\notag 
\end{align}
and the first inequality follows from $y_i\ge 0$ for $i=1,2,\ldots,p$ and $g(x,\tau)\le \theta(x)\ (\tau\in T)$.
Moreover, the second inequality is obtained from $\rho>\|y\|$ and $\theta(x)\ge 0$, and the last inequality is due to Proposition\,\ref{Prop:0329}.
We hence obtain 
\begin{align*}
\Phi_{\rho}(w+s\Delta w)-\left(\Phi_{\rho}(w)+\rho s\gamma\right)&\le -s\Delta x^{\top}B\Delta x+s\nu\Delta\psi+o(s).
\end{align*}
Now, since we can show $-\Delta x^{\top}B\Delta x+\nu\Delta\psi<0$ whenever $(\dx,\dV)\neq (0,O)$, 
it holds that  
\begin{equation}
\Phi_{\rho}(w+s\Delta w)\le \Phi_{\rho}(w)+\alpha s\left(-\Delta x^{\top}B\Delta x+\nu\Delta\psi\right)+\rho s\gamma \notag 
\end{equation}
for all sufficiently small $s>0$.
The proof is complete.
\end{proof}
\begin{Prop}\label{penal}
Suppose that Assumption~A-\ref{B3} holds.
There exist some $\bar{\rho}>0$ and $\bar{r}>0$
such that $\rho_r=\bar{\rho}$ for any $r\ge \bar{r}$.
\end{Prop}
\begin{proof}
The proof is easily obtained, and so omitted.
\end{proof}
Let $\bar{r}>0$ and $\bar{\rho}>0$ be as given in the above proposition.
The following proposition shows that the produced sequences are bounded. 
\begin{Prop}\label{property_bound}
Suppose that Assumption~A holds.
Then, we have the following:
\begin{enumerate}
\item\label{prop:bound}
$\left\{\Phi_{{\rho}_r}(x^r,V_r)\right\}$ is bounded from above.
\item\label{rnum:infF} $\liminf_{r\to \infty}\det F(x^r)>0$ and $\liminf_{r\to \infty}\det V_r>0$.
\item\label{prop:0820-w} 
$\{\xi(\ru[x][r])\}$ with $\xi(\cdot)$ defined by \eqref{eq:wx} is bounded.
\item\label{bound_v} $\{(y^r,z^r,V_r)\}$ is bounded. 
\item\label{pr_bound} $\{\rl[P][r]\}$ and $\{\rl[P][r]^{-1}\}$ are bounded.
\end{enumerate}
\end{Prop}
\begin{proof}
{\rm 1.} By the linesearch {rule}, 
we have, for $r\ge \bar{r}$,  
\begin{align*} 
\Phi_{{\bar{\rho}}}(x^r,V_r)&\le \Phi_{{\bar{\rho}}}(x^{r-1},V_{r-1})+\bar{\rho}\gamma_{r-1}\\
&= \Phi_{{\bar{\rho}}}(x^{\bar{r}},V_{\bar{r}})+\bar{\rho}\gamma_{\bar{r}}\sum_{i=\bar{r}}^r\beta_2^{i-\bar{r}}\\
&\le \Phi_{{\bar{\rho}}}(x^{\bar{r}},V_{\bar{r}})+\bar{\rho}\gamma_{\bar{r}}\sum_{i=\bar{r}}^{\infty}\beta_2^{i-\bar{r}}\\ 
&<\infty,
\end{align*}
where
the last inequality follows from $0<\beta_2<1$.
We thus get the boundedness of $\left\{\Phi_{\bar{\rho}}(x^r,V_r)\right\}$. \vspace{0.5em} \\
{\rm 2.} We first prove $\liminf_{r\to \infty}\det F(x^r)>0$. Assume to the contrary that
$\liminf_{r\to \infty}\det F(x^r)=0$. 
Notice that $\{\psi(x^r,V_r)\}$ is bounded from below by Lemma\,\ref{lem:0819}.
Also, notice that ${\chi}_{\bar{\rho}}(x^r)\to \infty$ as $r\to \infty$, where $\chi_{\bar{\rho}}(\cdot)$
is defined in \eqref{al:chi} with $\rho=\bar{\rho}$. This is because $\{x^r\}$ is bounded by assumption.
We then see that $\Phi_{\bar{\rho}}(x^r,V_r)\to \infty$ as $r\to \infty$, which contradicts item-\ref{prop:bound}. 
Hence we obtain $\liminf_{r\to \infty}\det F(x^r)>0$.

We next show $\liminf_{r\to \infty}\det V_r>0$.
For contradiction, 
we assume without loss of generality {that} $\lim_{r\to\infty}\det V_r=0$. Notice that $\{-\mu\log\det F(x^r)\}$ is bounded as
$\{x^r\}$ is bounded and $\liminf_{r\to \infty}\det F(x^r)>0$.
In addition, $F(x^r)\bullet V_r>0$ follows from $F(x^r),V_r\in S^m_{++}$. 
In view of these facts, we have {$\lim_{r\to\infty}\Phi_{\bar{\rho}}(x^r,V_r)=\infty$}. This contradicts item-\ref{prop:bound} again.   
Hence, we conclude $\liminf_{r\to \infty}\det V_r>0$. \vspace{0.5em}\\
%
{\rm 3.} Note that  $\liminf_{r\to \infty}\det F(x^r)>0$ by item-\ref{rnum:infF} and $\{F(x^r)\}\subseteq S^m_{++}$ 
is bounded since $\{x^r\}$ is bounded. Then, 
$\{F(x^r)^{-1}\}$ is also bounded and 
there exists some $M>0$ such that   
$\|F(x^r)^{-1}\|_F\le M$ for any $r\ge 0$. 
We then have  
\begin{equation*}
\|\xi(x^r)\|\le \sqrt{\sum_{i=1}^n\|F(x^r)^{-1}\|_F^2\|F_i\|_F^2}
           =\|F(x^r)^{-1}\|_F\sqrt{\sum_{i=1}^n\|F_i\|_F^2}\le M\sqrt{\sum_{i=1}^n\|F_i\|_F^2}
\end{equation*}
for all $r\ge 0$.
We thus obtain the desired result. \vspace{0.5em}\\
{\rm 4.} The boundedness of $\{\left(y^r,z^r\right)\}$ can be obtained from the boundedness of penalty parameters {$\rho_r$} (Assumption~A-\ref{B3}). We have only to show the boundedness of $\{V_r\}$.
For contradiction, suppose that $\{V_r\}$ is unbounded.
We may assume 
without loss of generality that
$\|V_r\|\to \infty$ as $r\to\infty$.
Let $X_r:=F(x^r)^{-\frac{1}{2}}V_rF(x^r)^{-\frac{1}{2}}$.
Denote 
the eigenvalues of $X_r\in S^m_{++}$ by $0<\lambda_1^r\le\lambda_2^r\le \cdots\le\lambda_m^r$.
Then, by the positive definiteness of $X_r$ and the boundedness of $F(x^r)^{-\frac{1}{2}}$
derived from item-\ref{rnum:infF}, it follows that $\lim_{r\to \infty}\lambda^r_m=\infty$.
We then obtain 
\begin{align}
\psi(x^r,V_r)&= {\rm Tr}(X_r)-\mu\log\det X_r\notag \\
                       & = \sum_{i=1}^m\left(\lambda_i^r-\mu\log\lambda_i^r\right)\rightarrow \infty \ (r\to \infty),\notag    
\end{align}
which together with the boundedness of $\{{\chi}_{\rho}(x^r)\}$ implies the unboundedness of $\{\Phi_{\rho}(x^r,V_r)\}$.
This contradicts item-\ref{prop:bound}.
We thus conclude that $\{V_r\}$ is bounded.\vspace{0.5em}\\
{\rm 5.} Since $\liminf_{r\to\infty}\det \rl[V][r]>0$ by item-\ref{rnum:infF}, we see that $\{\rl[V][r]^{-1}\}$ is bounded.
By Assumption~A-\ref{B2}, $\rl[P][r]$ is set to be $\rl[P][r]=F(\ru[x][r])^{-\frac{1}{2}}$ or $\rl[P][r]=W_r^{-\frac{1}{2}}$ with
$$
W_r=F(\ru[x][r])^{\frac{1}{2}}(F(\ru[x][r])^{\frac{1}{2}}\rl[V][r]F(\ru[x][r])^{\frac{1}{2}})^{-\frac{1}{2}}F(\ru[x][r])^{\frac{1}{2}}.
$$
Using these facts, it is not difficult to verify that $\{\rl[P][r]\}$ and $\{\rl[P][r]^{-1}\}$ are bounded. 
\end{proof}
Below, we additionally impose the following assumption on the sequence $\{B_r\}$: 
\begin{description}
\item[Assumption~B:]\label{B1}$\{B_{r}\}\subseteq S^n_{++}$ is a bounded sequence such that  
$$
\eta_2 I \succeq B_{r} \succeq \eta_1 I,\ \ r=0,1,2,\ldots
$$ 
for some $\eta_2>\eta_1>0$.
\end{description}
Assumption~B holds true if we choose the identity matrix as $B_r$ for all $r\ge 0$.
However, it is not obvious when the sequence $\{B_r\}$ suggested in Section~\ref{subsec:Br} becomes bounded.
In the next proposition, we prove that $\{B_r\}$ is bounded if $f$ is convex and Assumption~A holds.
\begin{Prop}
Suppose that $f$ and $g(\cdot,\tau)\ (\tau\in T)$ are twice continuously differentiable convex functions and Assumption~A holds.
Further, assume that the matrices $F_i\ (i=1,2,\ldots,n)$ are linearly independent in $S^m$.
Then, the sequence of matrices $B_r$ defined by either of the following formulas satisfies Assumption~B:
\begin{enumerate}
\item $B_r:=\nabla_{xx}^2L(x^r,y^r,z^r)$ for any $r\ge 0$,
\item  $B_r:=B(x^r,y^r,V_r)$ for any $r\ge 0$,
\end{enumerate} 
where $\nabla^2_{xx}L(x,y,z)$ and $B(x,y,V)$ are defined by \eqref{eq:hessian} and \eqref{eq:0913}, respectively. 
\end{Prop}
\begin{proof}
Denote $M_r:=\nabla^2f(x^r)+\int_T\nabla^2_{xx}g(x^r,\tau)dy^r(\tau)$ for any $r\ge 0$.
Note that $M_r\in S^m_{+}$ follows for each $r\ge 0$ from the convexity of $f$ and $g(\cdot,\tau)\ (\tau\in T)$
together with $y^r\in\M_+(T)$, and moreover $\{M_r\}$ is bounded, since $\{(x^r,y^r)\}$ is bounded by Proposition\,\ref{property_bound}. \vspace{0.5em}\\
{\rm 1.} Let $H_r:=-\nabla^2\log\det F(x^r)$ and denote $d{F}:=\sum_{i=1}^nd_iF_i$ for $d\in \R^n$. 
Then, the $(i,j)$-th entry of $H_r$ is represented as
$
F_i\bullet F(x^r)^{-2}F_j
$
for $1\le i,j\le m$, and we have 
\begin{equation*}
d^{\top}H_rd=\sum_{1\le i,j\le m}d_iF_i\bullet F(x^r)^{-2}d_jF_j=dF\bullet F(x^r)^{-2}dF. 
\end{equation*}
By the boundedness of $\{F(x^r)^{-1}\}\subseteq S^m_{++}$ and $\{F(x^r)\}\subseteq S^m_{++}$ 
from Proposition\,\ref{property_bound}
together with the linear independence of $F_1,F_2,\ldots,F_n$,
there exist some $c_1,c_2>0$ such that 
$$
c_1\le d^{\top}H_rd\le c_2
$$
for any $r\ge 0$ and $d\in \R^n$ with $\|d\|=1$.
Therefore, $\{H_r\}$ is uniformly positive definite and bounded.
Since $\nabla^2_{xx}L(x^r,y^r,z^r)=M_r+\mu H_r$, the sequence $\left\{\nabla^2L(x^r,y^r,z^r)\right\}$ satisfies Assumption~B.\vspace{0.5em} \\
{\rm 2.} Let $d\tilde{F}:=\sum_{i=1}^nd_i\tilde{F}_i$ for $d\in \R^n$. 
Recall that ${H}_P(x,V)$ is defined by \eqref{eq:HP} for $P\in \R^{m\times m}$.
It then holds that 
\begin{align}
d^{\top}{H}_{P_r}(x^r,V_r)d
&=\sum_{1\le i,j\le m}d_i\tilde{F}_i\mathcal{L}^{-1}_{\tilde{F}(x^r)}\mathcal{L}_{\tilde{V_r}}\tilde{F}_jd_j\notag \\
&=d\tilde{F}\bullet \mathcal{L}^{-1}_{\tilde{F}(x^r)}\mathcal{L}_{\tilde{V}}d\tilde{F}\notag \\
&=\mathcal{L}^{-1}_{\tilde{F}(x^r)}d\tilde{F}\bullet \mathcal{L}_{\tilde{V}_r}\mathcal{L}_{\tilde{F}(x^r)}\mathcal{L}^{-1}_{\tilde{F}(x^r)}d\tilde{F}\notag \\
&=\mathcal{L}^{-1}_{\tilde{F}(x^r)}d\tilde{F}\bullet(\tilde{F}(x^r)\circ \tilde{V}_r)\mathcal{L}^{-1}_{\tilde{F}(x^r)}d\tilde{F}\label{al:0916},
\end{align}
where the third equality is due to the symmetry of the linear operator $\mathcal{L}_{\tilde{F}(x^r)}$ and the last one holds since $\tilde{F}(x^r)$ and $\tilde{V}_r$ commute.
When the scaling matrix $P_r$ is the NT matrix, 
we obtain 
\begin{equation}
\eqref{al:0916}
=d\tilde{F}\bullet \tilde{V}_rd\tilde{F}\label{eq:0916-1}
\end{equation}
from $\tilde{F}(x^r)=I$. 
On the other hand, when $P_r$ is the HRVW/KSH/M matrix, we get
\begin{equation}
\eqref{al:0916}
=
\mathcal{L}^{-1}_{\tilde{F}(x^r)}d\tilde{F}\bullet 
\tilde{F}(x^r)^2\mathcal{L}^{-1}_{\tilde{F}(x^r)}d\tilde{F}
\label{eq:0916-2}
\end{equation}
since $\tilde{F}(x^r)=\tilde{V}_r$.
By noting the boundedness of $\{P_r\}$ and $\{P_r^{-1}\}$ from Proposition\,\ref{property_bound},
the sequences 
$\{\tilde{F}(x^r)\}$, $\{\tilde{F}(x^r)^{-1}\}$, and $\{\tilde{V}_r\}$ are bounded.
Using these facts together with the linear independence of $F_1,F_2,\ldots,F_n$,
the above expressions \eqref{eq:0916-1} and \eqref{eq:0916-2} yield that there exist some $c_3,c_4>0$ such that 
$$
c_3\le d^{\top}{H}_{P_r}(x^r,V_r)d\le c_4
$$
for any $r\ge 0$ and $d\in \R^n$ with $\|d\|=1$.
Therefore, $\{{H}_{P_r}(x^r,V_r)\}$ is uniformly positive definite and bounded.
Since $B(x^r,y^r,V_r)=M_r+{H}_{P_r}(x^r,V_r)$, the sequence $\left\{B(x^r,y^r,V_r)\right\}$ satisfies Assumption~B.
\end{proof}

We next present the following proposition concerning $\{(\Delta x^r,\Delta V_r)\}$.
\begin{Prop}\label{prop:bd:dx}
Suppose that Assumptions~A and B hold.
Then, we have the following:
\begin{enumerate}
\item\label{wpx} $\{\Delta x^r\}$ is bounded;
\item $\{\Delta V_r\}$ is bounded.
\end{enumerate}
\end{Prop}
\begin{proof}
{\rm 1.} Let $x_{\rm f}\in \R^n$ be a feasible point for \eqref{lsisdp}, i.e., a point satisfying $Gx_{\rm f}=h$ and $g(x_{\rm f},\tau)\le 0\ (\tau\in T)$.
Then, 
according to Assumptions~A-\ref{B4}, B, and item-\ref{prop:0820-w}
of Proposition\,\ref{property_bound},
there exists some positive constant $M>0$ by which the following three sequences are bounded from above:
$\left\{\|\xi(x^r)\|\right\}$, $\left\{\|\nabla f(\ru[x][r])\|\right\}$, and
$$
\left\{\nabla f(x^r)^{\top}(x_{\rm f}-x^r)+\frac{1}{2}(x_{\rm f}-x^r)^{\top}B_r(x_{\rm f}-x^r)-\mu \xi(x^r)^{\top}(x_{\rm f}-x^r)\right\}.$$
Then, 
since $x_{\rm f}-x^r$ is feasible to SIQP\,\eqref{SIQP} with $x=x^r$ and $B=B_r$ for any $r\ge 0$ and 
$B_r\succeq \eta_1 I$ (Assumption~B), 
we have 
\begin{align*}
&-M\|\Delta x^r\|+\frac{\eta_1}{2}\|\Delta x^r\|^2-\mu M\|\Delta x^r\|\\
&\le \nabla f(\ru[x][r])^{\top}\Delta x^r+\frac{1}{2}(\Delta x^r)^{\top}B_r\Delta x^r-\mu \xi(x^r)^{\top}\Delta x^r\\
&\le \nabla f(\ru[x][r])^{\top}(x_{\rm f}-x^r)+\frac{1}{2}(x_{\rm f}-x^r)^{\top}B_r(x_{\rm f}-x^r)-\mu \xi(x^r)^{\top}(x_{\rm f}-x^r)\\
&\le M,
\end{align*}
from which it is easy to see the boundedness of $\{\|\Delta x^r\|\}$.\vspace{0.5em}\\
{\rm 2.} Recall that
$\{F(x^r)\}$, $\{V_r\}$, $\{P_r\}$, and the sequences of their inverse matrices are bounded
by the previous statements and assumptions.
Then, 
the scaled sequences
$\{\tilde{F}(x^r)\}$, 
$\{\tilde{V}_r\}$, 
$\{\tilde{F}(x^r)^{-1}\}$, and $\{\tilde{V}_r^{-1}\}$ are all bounded.
We thus find that the sequences of linear operators $\{\mathcal{L}_{\tilde{F}(x^r)}\}$ and $\{\mathcal{L}_{\tilde{F}(x^r)}^{-1}\}$ are bounded. 
By these observations and the Newton equations $\tilde{V}_r+\Delta \tilde{V}_r=\mu \tilde{F}(x^r)^{-1}
-\mathcal{L}_{\tilde{F}(x^r)}^{-1}\left(\sum_{i=1}^n\Delta x_i \tilde{F}_i\circ \tilde{V}_r\right)$ (see \eqref{eq:scaled_new}),
$\{\Delta \tilde{V}_r\}$ is bounded.
Then, by using 
$\Delta V_r=P_r\Delta \tilde{V}_rP_r^{\top}$
and the boundedness of $\{P_r\}$ again, 
we see that $\{\Delta {V}_r\}$ is bounded. 
\end{proof}}

\begin{Prop}\label{prop:conv}
Suppose that Assumptions~A and B hold.
Then, $\{\Phi_{\rl[\rho][r]}(\ru[x][r],\rl[V][r])\}$ is convergent.
\end{Prop}
\begin{proof}
By the line search procedure, we have, for any $r\ge \bar{r}$, $\rho_r=\bar{\rho}$ and
\begin{equation}
\Phi_{{\bar{\rho}}}(\ru[x][r+1],\rl[V][r+1])\le \Phi_{\bar{\rho}}(\ru[x][r],\rl[V][r])+s\bar{\rho}\rl[\gamma][r].\notag  
\end{equation}
Notice that $\sum_{r=0}^{\infty}\rl[\gamma][r]=\gamma_0\sum_{r=0}^{\infty}\beta^r<\infty$.
Also, notice that $\{\Phi_{\bar{\rho}}(x^r,V_r)\}$ is bounded below, since $\left\{\left(x^r,V_r\right)\right\}$ is bounded by Assumption~A-\ref{B4}
and 
item\,\ref{bound_v} in Proposition\,\ref{property_bound}.
In view of these observations, $\{\Phi_{\bar{\rho}}(x^r,V_r)\}$ is a convergent sequence. 
\end{proof}

\begin{Prop}\label{prop:0719}
Suppose that Assumptions~A and B hold.
Then, $\Delta x^r\to 0$ and $\Delta V_r\to O$ as $r\to \infty$.
\end{Prop}
\begin{proof}
From Propositions\,\ref{property_bound}, \ref{prop:bd:dx}, and Assumption~A-\ref{B4},
$\left\{\left(x^r,V_r,\Delta x^r,\Delta V_r,B_r\right)\right\}$ is bounded and has at least one accumulation point, say $(x^{\ast},V_{\ast},\Delta x^{\ast},\Delta V_{\ast},B_{\ast})\in \R^n\times S^m_{++}\times \R^n\times S^m\times S^n_{++}$. 
Without loss of generality, we may suppose that 
$$
\lim_{r\to \infty}(x^r,V_r,\Delta x^r,\Delta V_r,B_r)=\left(x^{\ast},V_{\ast},\Delta x^{\ast},\Delta V_{\ast},B_{\ast}\right).
$$
Since $\{\Phi_{\bar{\rho}}(x^r,V_r)\}$ is convergent according to Proposition\,\ref{prop:conv},
\begin{equation}
s_r\left(-(\Delta x^r)^{\top}B_r\Delta x^r
+
\psi^{\prime}(x^r,V_r;\Delta x^r,\Delta V_r)\right)\to 0
\end{equation} 
as $r\to \infty$, from which together with $B_r\in S^n_{++}$ and 
$\psi^{\prime}(x^r,V_r;\Delta x^r,\Delta V_r)\le 0$
we have 
\begin{equation}
\lim_{r\to \infty}s_r(\Delta x^r)^{\top}B_r\Delta x^r=0,\ \lim_{r\to \infty}s_r\psi^{\prime}(x^r,V_r;\Delta x^r,\Delta V_r)=0\label{sono0}
\end{equation} 
 as $r\to \infty$.
If $\liminf_{r\to\infty}s_r>0$ holds, we can easily derive that $\Delta x^{\ast}=0$ and $\Delta V_{\ast}=O$, and the proof is complete.
Suppose $\liminf_{r\to\infty}s_r=0$. 
Without loss of generality, we may assume that 
$\lim_{r\to\infty}s_r=0$.
It follows
from the linesearch rule that 
\eqref{eq:armijo-rule} does not hold with $s=s_r/\beta$, namely,
\begin{align}
&\beta\left(\Phi_{\bar{\rho}}\left(x^r+\frac{s_r}{\beta}\Delta x^r,V_r+\frac{s_r}{\beta}\Delta V_r\right)-\Phi_{\bar{\rho}}\left(x^r,V_r\right)\right)\notag \\
>&-\alpha s_r(\Delta x^r)^{\top}B_r\Delta x^r+\nu {\alpha}s_r\psi^{\prime}(x^r,V_r;\Delta x^r,\Delta V_r)+
\bar{\rho}s_r\gamma_r\label{sono1}
\end{align}
for any $r$. 
Dividing both sides of \eqref{sono1} by $s_r>0$ and letting $r\to \infty$ yield
\begin{align}
\Phi^{\prime}_{\bar{\rho}}(x^{\ast},V_{\ast};\Delta x^{\ast},\Delta V_{\ast})
&=\chi^{\prime}_{\bar{\rho}}(x^{\ast};\Delta x^{\ast})+\nu\psi^{\prime}(x^{\ast},V_{\ast};\Delta x^{\ast},\Delta V_{\ast})\notag\\
&\ge\alpha\left(-(\Delta x^{\ast})^{\top}B_{\ast}\Delta x^{\ast}+\nu\psi^{\prime}(x^{\ast},V_{\ast};\Delta x^{\ast},\Delta V_{\ast})\right).\label{sono2}
\end{align}
On the other hand, it follows from \eqref{a3} that
\begin{align*}
&{\Phi_{\bar{\rho}}(x^r+s_r\Delta x^r,V_r+s_r\Delta V_r)-\Phi_{\bar{\rho}}(x^r,V_r)}\\
&\hspace{5em}\le -s_r(\Delta x^r)^{\top}B_{r}\Delta x^{r}+
\nu s_r\psi^{\prime}(x^r,V_r;\Delta x^r,\Delta V_r)
+s_r\bar{\rho}\gamma_r+o(s_r).
\end{align*}
Dividing both sides by $s_r$ and letting $r\to\infty$
yield
\begin{equation}
\Phi^{\prime}_{\bar{\rho}}(x^{\ast},V_{\ast};\Delta x^{\ast},\Delta V_{\ast})
\le -(\Delta x^{\ast})^{\top}B_{\ast}(\Delta x^{\ast})^{\top}+
\nu\psi^{\prime}(x^{\ast},V_{\ast};\Delta x^{\ast},\Delta V_{\ast}).\label{sono3}
\end{equation}
Combining \eqref{sono2} and \eqref{sono3}, we have 
\begin{equation}
\alpha\left(-(\Delta x^{\ast})^{\top}B_{\ast}\Delta x^{\ast}
+\nu\psi^{\prime}(x^{\ast},V_{\ast};\Delta x^{\ast},\Delta V_{\ast})
\right)\le -(\Delta x^{\ast})^{\top}B_{\ast}\Delta x^{\ast}+\nu\psi^{\prime}(x^{\ast},V_{\ast};\Delta x^{\ast},\Delta V_{\ast}).\label{eq:sono4}
\end{equation}
By the positive definiteness of $B_{\ast}$, we have $(\Delta x^{\ast})^{\top}B_{\ast}\Delta x^{\ast}\ge 0$.
From Lemma\,\ref{lem_bp},
we can deduce $\psi^{\prime}(x^{\ast},V_{\ast};\Delta x^{\ast},\Delta V_{\ast})\le 0$.
Hence, the relation\,\eqref{eq:sono4} together with $\alpha\in (0,1)$ yields that $\Delta x^{\ast}=0$ and $\psi^{\prime}(x^{\ast},V_{\ast};\Delta x^{\ast},\Delta V_{\ast})=0$. Moreover, we obtain $\Delta V_{\ast}=O$ from Lemma\,\ref{lem:0820}.
\end{proof}
Using the above propositions, we have the following convergence theorem.
\begin{Thm}
Suppose that Assumptions~A and B hold.
Then, the generated sequence
$\left\{\left(x^r,y^r,z^r,V_r\right)\right\}$ is bounded.
Furthermore, {any $\wast$-accumulation point $(x^{\ast},y^{\ast},z^{\ast},V_{\ast})$ of $\{(x^r,y^r,z^r,V_r)\}$} is a KKT point for the SIPLOG\,\eqref{lsisdp}.
\end{Thm}
{\begin{proof}
The first claim follows from Assumption~A\,\ref{B4} and item\,\ref{bound_v} of Proposition\,\ref{property_bound} immediately.

We show the second-half. 
Recall that any bounded sequence in $\M(T)$ has at least one weak* accumulation point
and one can extract a subsequence weakly* converging to that point. 
Hence, 
we may assume that the entire sequence $\{(x^r,y^r,z^r,V_r)\}$ weakly* converges to $(x^{\ast},y^{\ast},z^{\ast},V_{\ast})$ without loss of generality.                        
Since 
\begin{equation}
\Delta x^{r-1}\to 0,\ \Delta V_{r-1}\to O\label{eq:0719-1356}
\end{equation}
as $r\to \infty$ by Proposition\,\ref{prop:0719} and $s_{r-1}\in [0,1]$ for each $r$, 
we see that 
\begin{equation}
\lim_{r\to \infty}(x^{r-1},V_{r-1}) = \lim_{r\to \infty}
\left(x^{r}-s_{r-1}\Delta x^{r-1},V_{r-1}-s_{r-1}\Delta V_{r-1}\right) =(x^{\ast}, V_{\ast}),\label{eq:0719-1352}
\end{equation}
which together with ${\rm w}^{\ast}\mbox{-}\lim_{r\to \infty}y^r = y^{\ast}$ implies
\begin{equation}
\lim_{r\to \infty}
\int_T\left(g(x^{r-1},\tau)+\nabla_xg(x^{r-1},\tau)^{\top}\Delta x^{r-1}\right)dy^r(\tau)=
\int_Tg(x^{\ast},\tau)dy^{\ast}(\tau).\label{eq:0719-1353}
\end{equation}
Moreover, since $\{P_r\}$ and $\{P_r^{-1}\}$ are both bounded by item~\ref{pr_bound} of Proposition\,\ref{property_bound}, 
\eqref{eq:0719-1356} implies 
\begin{equation}
\lim_{r\to \infty}\Delta \tilde{V}_{r-1}=
\lim_{r\to \infty}
P_{r-1}^{\top}\Delta {V}_{r-1}P_{r-1}=O.\label{eq:0719_1455}
\end{equation}
In addition, there exists an accumulation point $P_{\ast}\in \R^{m\times m}$ of $\{P_r\}$.
Without loss of generality, we can suppose that $\lim_{r\to \infty}P_r = P_{\ast}$.
Then, it holds that 
\begin{equation}
\lim_{r\to \infty}\left(
\tilde{F}(x^r),\tilde{V}_r
\right)
=\left(P_{\ast}F(x^{\ast})P_{\ast}^{\top},
P_{\ast}^{-\top}V_{\ast}P_{\ast}^{-1}
\right).\label{eq:0719-1706}
\end{equation}

From \eqref{al:opt} and \eqref{eq:scaled_new}, we have, for each $r$, 
\begin{align}\label{al:opt2}
&\nabla f(x^{r-1})+B\Delta x^{r-1}-\mu \xi(x^{r-1})+\int_T\nabla_xg(x^{r-1},\tau)dy^r(\tau)+G^{\top}z^{r}=0,\notag \\
&g(x^{r-1},\tau)+\nabla_xg(x^{r-1},\tau)^{\top}\Delta x^{r-1}\le 0\ \ (\tau\in {\rm supp}(y^r)),\\
&\int_T\left(g(x^{r-1},\tau)+\nabla_xg(x^{r-1},\tau)^{\top}\Delta x^{r-1}\right)dy^r(\tau)=0,\ G(x^{r-1}+\Delta x^{r-1})=h,\notag \\
&\max_{\tau\in T}\left(
g(x^{r-1},\tau)+\nabla_xg(x^{r-1},\tau)^{\top}\Delta x^{r-1}
\right)_+\le \gamma_{r-1},\ y^{r}\in \mathcal{M}_+(T),\notag \\
&\left(\tilde{F}(x^{r-1})+\sum_{i=1}^n\Delta x_i^{r-1}\tilde{F}_i\right)\circ\tilde{V}_{r-1}+
\tilde{F}(x^{r-1})\circ\Delta \tilde{V}_{r-1}=\mu I,\ \tilde{F}(x^{r-1})\in S^m_{+}, \tilde{V}_{r-1}\in S^m_{+} \notag 
\end{align}
Note \eqref{eq:0719-1356}--\eqref{eq:0719-1706}, and $\gamma_{r-1}\to 0\ (r\to \infty)$.
Then, by letting $r\to \infty$ in \eqref{al:opt2}, we conclude that 
$(x^{\ast},y^{\ast},z^{\ast},V_{\ast})$ is a KKT point of the SIPLOG\,\eqref{lsisdp}.
\end{proof}}

\section{Numerical experiments}\label{sec:4}
In this section, we conduct some numerical experiments to demonstrate the efficiency of the interior point SQP-type algorithm (Algorithm~1). We consider two kinds of SIPLOGs with a one-dimensional index set $T$ of the form $[T_{\rm min},T_{\rm max}]$: The first one is a linear SIPLOG (LSIPLOG) where the functions $f$ and $g(\cdot,\tau)\ (\tau\in T)$ are affine; the second one is a nonlinear SIPLOG (NSIPLOG) with $f$ being a quartic objective function that is not convex in general.
 In this experiment, we compute KKT points of these problems for various values of $\mu$ and $m$. 
Throughout the section, 
{to evaluate the distance of $(x,y,z,V)$ to the set of KKT points of the SIPLOG, 
we use the function $R:\R^n\times \mathcal{M}(T)\times 
\R^s
\times S^m\to \R$ with the parameter $\mu>0$ 
defined by 
\begin{equation}
R(x,y,z,V):=
\sqrt{
\theta(x)^2+
\|\phi_1(x,{y},z,{V})\|^2+\phi_2(x,y)^2+
\|\phi_3(x,V)\|^2+\|Gx-h\|^2},\notag 
\end{equation}
where
$\theta(x):=\max_{\tau\in T}\,\left(g(x,\tau)\right)_+$,
$\phi_1(x,{y},{V}):=\nabla f(x)+{\displaystyle \int_T\nabla_xg(x,\tau)\ydt} -(F_i\bullet V)_{i=1}^n+G^{\top}z$,
$\phi_2(x,y):=\int_Tg(x,\tau)\ydt$, and $\phi_3(x,V):=\|F(x)\circ V-\mu I\|$.
Note that a point $(x,y,z,V)$ satisfying $R(x,y,z,V)=0$ with $F(x)\in S^m_{+}$ and $V\in S^m_+$ is a KKT point of the SIPLOG\,\eqref{lsisdp}.}
We identify a symmetric matrix variable $X\in S^m$ with a vector variable $x:=(x_{11},x_{12},\ldots,x_{1m},x_{12},x_{22},\ldots,x_{mm})^{\top}\in\R^{\frac{m(m+1)}{2}}$ through
$$
X=\begin{pmatrix}
             x_{11}&x_{12}&\ldots&x_{1m}\\
             x_{12}&x_{22}&\ldots&x_{2m}\\
           \vdots&\vdots&\ddots&\vdots\\
             x_{1m}&x_{2m}&\ldots&x_{mm}      
    \end{pmatrix}.
$$ 
The program is coded in MATLAB~R2012a and run on a machine with Intel(R) Xeon(R) CPU E5-1620 v3@3.50GHz and 10.24GB RAM. 
The actual implementation is as follows: 
To compute the values of the merit functions $R$ and $\Phi_{\rho}$,
we need to solve $\max_{\tau\in T}g({x},\tau)$.
For this purpose, we apply Newton's method%
\footnote{
There is no theoretical guarantee for the global optimality of $\tau$ gained in this way, but practically 
we may expect that such a $\tau$ is a global optimum.} combined with projection onto $T$
for the problem $\max_{\tau\in T}g({x},\tau)$ with a starting point ${\tau}\in \mathop{\rm argmax}\{g(x,s)\mid s=s_1,s_2,\ldots,s_{N+1}\}$,
where 
$$
s_i:=T_{\rm min} + \frac{(i-1)}{N}(T_{\rm max}-T_{\rm min})\ \ (i\in \{1,2,\ldots,N+1\}),\
N:=100.$$
In Step~0, we set
$$
\alpha=10^{-3},\ \beta_1=0.95,\ \beta_2=0.5,\ \sigma=0.95,\ \delta=1,\ \nu=1,\ \rl[\rho][0]=100,\ \rl[\gamma][0]=0.1.
$$ 
We choose starting points as $X^0 = m^{-1}I, y^0 = 0, z^0 = 0$, and $V_0=\mu I$.
In Step~1, we terminate the algorithm if the value of the function $R$ is less than $10^{-6}$,
In Step~2, 
we used the NT or HRVW/KSH/M (H.K.M) matrices as a scaling matrix $P_r$. 
As the matrix $B_r$, we set $B_r=\nabla^2_{xx}L_2(x^r,y^r)+H_{P_r}$, where the function $L_2$
is defined in \eqref{eq:0615-0916} and the matrix  
$H_{P_r}$ is defined by
\eqref{eq:HP} with $P=P_r$.
For the case of the NSIPLOG, we modified $B_P$ by lifting its negative eigenvalue to $1$ to assure $B_P\in S^m_{++}$.
In Step~3, we use the exchange method described in Section~\ref{sec:4} for finding a solution of the system\,\eqref{al:opt}.
We solve QPs by Matlab solver \texttt{fmincon} in Step~1 of the exchange method.
In Step~6, for the sake of numerical stability, we set $\gamma_{r+1}:=\max(10^{-8},\beta_2\gamma_{r})$.
\subsection{Linear SIPLOGs}
In this section, we compute a KKT point of the following LSIPLOG for various values of $\mu$.
This problem is obtained by slightly modifying 
the semi-infinite eigenvalue optimization problem solved in \cite{li2004solution}:
\begin{align}
\begin{array}{rcl}
\displaystyle{\mathop{\rm Maximize}_{X\in S^m}}& &A_0\bullet X-\mu\log\det X\\
\mbox{subject to}& & A(\tau)\bullet X\ge 0\ (\tau\in T)\\
            & & I\bullet X = 1\\ 
            & & X\in S^m_{++},
\end{array}\label{eig_semi}
\end{align}
where $A_0\in S^m$ and $A:T\to S^m$ is a symmetric matrix valued function 
whose elements are $q$-th order polynomials in $\tau$.  
We set $T=[0,1]$, i.e., $T_{\rm min}=0$ and $T_{\rm max}=1$, $(A(\tau))_{i,j}=\sum_{l=0}^qa_{i,j,l}\tau^l$ for $1\le i,j\le m$, and $q=9$.

We choose all entries of $A_0$ and the {coefficients} $a_{i,j,l}$ {in $A(\tau)$} from the interval $[-1,1]$ randomly. 
Among {those} generated data {sets}, we use only data such that the semi-infinite constraint
{includes at least one active constraint} 
at {an} optimum of \eqref{eig_semi}.
Specifically, for each data set, we compute an optimum, say $\tilde{X}$, of the SIPLOG obtained by removing the semi-infinite constraints.
If 
$\min_{1\le i\le 21}A\left(T_{\rm min} + \frac{i-1}{20}(T_{\rm max}-T_{\rm min})\right)\bullet \tilde{X}\le -10^{-3}$, which implies that $\tilde{X}$ does not satisfy the semi-infinite constraints, we adopt it as a valid data set.

In the above manner, we generated 10 test problems for each $(m,\mu)\in \{10,25\}\times \{1,10^{-5}\}$ and
applied the algorithm for solving the generated problems.
All instances were successfully solved. 
{We show the obtained results in Tables~\ref{ta1} and \ref{ta2}, where 
``time(s)'', ``$R^{\ast}$'', ``{$\sharp$QP}'', and ``$\sharp$ite'' {stand for}
the average running time in seconds, the average value of $R$ at the solution output by the algorithm, the average number of QPs solved per run, 
and the average number of iterations, respectively.
Moreover, ``H.K.M (resp. NT)''
{means} that the H.K.M (resp. NT) matrix is used as a scaling matrix $P_r$ in Step~2. 
We observe that $\sharp\mbox{ite}$ {tends} to increase as $m$ {increases}.
For example, Table\,\ref{ta1} {shows that}, when the H.K.M scaling matrix is used, 
$\sharp \mbox{ite}$ is 9.60 for $m=10$ while it is 13.40 for $m=25$. 
A similar tendency can be {observed} between the values of $\sharp\mbox{ite}$ and $\mu$.
Actually, for the case of $10(\mbox{H.K.M})$, we find that 
$\sharp\mbox{ite}$ {is} 9.6 for $\mu=1$
while it grows up to 33.2 for $\mu = 10^{-5}$. 
This phenomenon might be caused because a solution of the LSIPLOG 
approaches the boundary of the semi-definite region as $\mu$ decreases.
As the next observation, we see that $\frac{\sharp{\rm QP}}{\sharp{\rm ite}}$ lies between 1 and 2 in each row of Tables~\ref{ta1} and \ref{ta2}.
This indicates that, in the exchange method used in Step~2,  
a solution satisfying the conditions\,\eqref{al:opt} {were} found after solving only one or two QPs on average. 
{Finally, it may be worth mentioning that,
for many instances,
we observed superlinear-like convergence of 
the value of the function $R$ to 0
in a last stage of iterations.}
\begin{table}[h]
\centering
\small
\begin{tabular}{|c|c|c|c|c|}\hline
  $m$ (H.K.M./NT)                & time(s)&  $R^{\ast}$& {$\sharp$ QP}&$\sharp$ ite                        \\     \hline
10 (H.K.M.)                & 0.15        &  $6.35 \cdot 10^{-8}$  &11.7      &9.60   \\ 
25 (H.K.M.)                & 0.53       &  $ 1.77\cdot 10^{-7}$  & 15.9     & 13.40    \\ \hline 
10 (NT)                &     0.09         &   $ 8.40 \cdot 10^{-8}$      &13.0     &10.20      \\ 
25 (NT)                &     0.28         &    $1.76  \cdot 10^{-7}$     &17.8     & 14.00   \\ \hline
\end{tabular}
\caption{Results for the LSIPLOG with $\mu=1$}
\label{ta1}
\end{table}
\begin{table}
\centering
\small
\begin{tabular}{|c|c|c|c|c|}\hline
  $m$ (H.K.M./NT)                & time(s) &  $R^{\ast}$& {$\sharp$ QP}&$\sharp$ ite                        \\     \hline
10 (H.K.M.)                & 0.54        &  $1.44\cdot 10^{-7}$  &  45.9    &33.2   \\ 
25 (H.K.M.)                & 1.83         &  $7.88\cdot 10^{-8}$  &  42.4   & 37.9  \\ \hline 
10 (NT)                &     0.48          & $1.22\cdot 10^{-7}$    & 26.7    & 18.3     \\ 
25 (NT)                &      0.91         &   $1.62\cdot 10^{-7}$  &  23.4  & 19.1   \\ \hline
\end{tabular}
\caption{Results for the LSIPLOG with $\mu=10^{-5}$}
\label{ta2}
\end{table}

%
%
}
\subsection{Nonconvex SIPLOGs}
Next, we solve the following SIPLOG: 
\begin{align}
\begin{array}{rcl}
\displaystyle{\mathop{\rm Minimize}_{x\in \R^{\frac{m(m+1)}{2}}}}& &\frac{1}{2}x^{\top}Mx+c^{\top}x+\omega{\|x\|^4}-\mu\log\det (X+\kappa I )\\
\mbox{subject to}& & a(\tau)^{\top}x\le b(\tau)\ (\tau\in T)\\
                         & & X + \kappa I\in S^m_{++}
\end{array}\label{eig_semi}
\end{align}
where 
$\kappa >0$,
$a(\tau) : = (1,\tau,\tau^2,\tau^3,\ldots,\tau^{n-1})^{\top}\in \R^n$ with $n:=m(m+1)/2$ and
$b(\tau):=\sum_{i=1}^n\tau^{2i} + \sin(9\pi \tau)+2$.
All elements of $M\in S^n$ and $c\in \R^n$ are randomly generated from the interval $[-1,1]$. 
The objective function is not convex in general but coercive in the sense that 
$f(x)\to \infty$ as $\|x\|\to \infty$, and thus the problem is guaranteed to have at least one local optimum.
We set $T=[0,1]$ and $\kappa = \omega=0.01$, and generated 10 problems for each $(m,\mu)\in \{10, 25\}\times \{1,10^{-3}\}$.
We applied the algorithm for solving those problems. 

We show the obtained results in Tables~\ref{ta3} and \ref{ta4}, where each column means the same as in Tables~\ref{ta1} and \ref{ta2}.
Compared with the LSIPLOG, there are more 
significant differences between the results for $\mu=1$ and $\mu=10^{-3}$. 
Specifically, when using the H.K.M scaling matrix with $m=25$, 
$(\sharp {\rm ite},\sharp {\rm QP})$ is 
$(329.9,424.4)$ for $\mu =10^{-3}$, while 
$(\sharp {\rm ite},\sharp {\rm QP})$ is $(16.0,23.4)$ for $\mu =1$.
For the case where the NT scaling matrix is used, we also observe big differences between the results for
$\mu = 1$ and $\mu=10^{-3}$.
However, the NT scaling seems to exhibit more stable behavior than the H.K.M scaling.
In fact, when $m=25$, 
time(s) for the H.K.M changes drastically from $1.29$ to $28.4$, while time(s) for the NT increases from $1.19$ to $7.32$.
As for $\sharp {\rm ite}$ and $\sharp {\rm QP}$, a similar tendency is observed. 
\begin{table}[h]
\centering
\small
\begin{tabular}{|c|c|c|c|c|}\hline
  $m$ (H.K.M./NT)                & time(s) &  $R^{\ast}$& {$\sharp$ QP}&$\sharp$ ite                        \\     \hline
10 (H.K.M.)                &   0.22      &  $8.24\cdot 10^{-7}$  &16.3  &9.90  \\ 
25 (H.K.M.)                &    1.29   &  $2.02\cdot 10^{-7}$  & 23.4  &16.0    \\ \hline 
10 (NT)                &        0.19      & $9.28\cdot 10^{-7}$    &14.9  & 10.0    \\ 
25 (NT)                &         1.19      &   $4.04\cdot 10^{-7}$  & 24.3   & 13.2    \\ \hline
\end{tabular}
\caption{Results for the NSIPLOG with $\mu =1$}
\label{ta3}
\end{table}
\begin{table}
\centering
\small
\begin{tabular}{|c|c|c|c|c|}\hline
  $m$ (H.K.M./NT)                & time(s) &  $R^{\ast}$& {$\sharp$ QP}&$\sharp$ ite                        \\     \hline
10 (H.K.M.)                & 0.98        &  $1.78\cdot 10^{-7}$  &  71.8    &54.5   \\ 
25 (H.K.M.)                &  28.40         &  $4.69\cdot 10^{-7}$  & 422.4     & 329.9    \\ \hline 
10 (NT)                &     0.45           & $1.77\cdot 10^{-7}$    & 40.2    & 22.6     \\ 
25 (NT)                &      7.32          &   $3.48\cdot 10^{-7}$  & 162.8    &  75.5   \\ \hline
\end{tabular}
\caption{Results for the NSIPLOG with $\mu = 10^{-3}$}
\label{ta4}
\end{table}

\section{Conclusion}\label{sec:conclusion}
In this paper, we have proposed the interior point SQP-type method (Algorithm~1) for finding a KKT point for the SIPLOG\,\eqref{lsisdp}.
In this method, we generate a sequence of inexact KKT points of semi-infinite quadratic programs approximating the SIPLOG\,\eqref{lsisdp}. 
We further solve scaled Newton equations to generate a 
NT or HRVW/KSH/M search direction in the dual matrix space. 
We have shown that any weak* accumulation point of a produced sequence is a KKT point for the SIPLOG under some assumptions.
To examine the efficiency of the proposed algorithm,
we conducted some numerical experiments, in which we solve the SIPLOG\,\eqref{lsisdp} for various values of the barrier parameter $\mu$.
From the numerical results, we observed that the proposed algorithm performed well for finding a KKT point of the SIPLOG.
As a future work, we will develop a path-following method for solving the SISDP\,\eqref{lsisdp2} based on Algorithm~1.

\end{document}